\documentclass[reqno,final,11pt]{amsart}
\usepackage{natbib}  
\usepackage{fancyhdr} 
\usepackage{color} 
\usepackage{hyperref} 
\usepackage{graphicx} 
\usepackage{geometry}
\usepackage[utf8]{inputenc}
\usepackage[T1]{fontenc}
\usepackage{verbatim}

\usepackage{tikz}
\usetikzlibrary{calc,shapes,backgrounds,arrows}
\usepackage{amssymb}

\definecolor{aleacolor}{rgb}{0.16,0.59,0.78}

\hypersetup{
breaklinks,
colorlinks=true,
linkcolor=aleacolor,
urlcolor=aleacolor,
citecolor=aleacolor}

\renewcommand{\cite}{\citet}

\theoremstyle{plain}
\newtheorem{theorem}{Theorem}[section]                                          
                          
\newtheorem{lemma}[theorem]{Lemma}

\theoremstyle{definition}
\newtheorem{definition}[theorem]{Definition}
\theoremstyle{remark}
\newtheorem{remark}[theorem]{Remark}

\makeatletter \@addtoreset{equation}{section} \makeatother
\renewcommand\theequation{\thesection.\arabic{equation}}
\renewcommand\thefigure{\thesection.\arabic{figure}}
\renewcommand\thetable{\thesection.\arabic{table}}

\newtheorem{condition}[theorem]{\bf Condition}
\newcommand{\N}{\mathbb{Z}_{+}}

\newcommand{\R}{\mathbb{R}}
\newcommand{\PP}{\mathbb{P}}
\newcommand{\E}{\mathbb{E}}
\newcommand{\ve}{\varepsilon}
\newcommand{\xe}{\frac{x}{\varepsilon\varepsilon_{\gamma}^{1/2}}}
\newcommand{\ti}{\theta(\frac{x}{\left\vert x\right\vert})} 

\begin{document}

\title[LDP for Peano Phenomenon with Homogeneous Drift.]{Large Deviations and the Peano Phenomenon in Stochastic Differential Equations with Homogeneous Drift.}

\author{Bermolen~Paola}
\address{Universidad de la República, Uruguay}
\email{paola@fing.edu.uy} 

\author{Goicoechea~Valeria}
\email{vgoicoechea@fing.edu.uy} 

\author{León~José Rafael}
\email{rlramos@fing.edu.uy} 

\thanks{This work has been partially supported by ANII, FCE-3-2024-1-180711. }
\subjclass[60F10]{60H10, 34F05, 60J35} 

\begin{abstract}
We consider a diffusion equation in  $\R^d$ with drift equal to the gradient of a homogeneous potential of degree $1+\gamma$, with $0<\gamma<1$, and local variance equal to $\ve^2$ with $\ve\to0$. The associated deterministic system for $\ve=0$ has a potential that is not a Lipschitz function on a neighbourhood of the origin. Therefore, an infinite number of solutions exist, known as the Peano phenomenon.  In this work, we study first- and second-order large deviations for a noisy system, generalizing previous results for the specific potential $b(x)=x |x|^{\gamma-1}$. For first-order large deviations, we recover the rate function from the well-known work of Freidlin and Wentzell. For the second-order large deviation, we use a refinement of Carmona-Simon
bounds for the eigenfunctions of a Schrödinger operator, and prove that the exponential behavior of the process depends only on the ground state of such an operator. Moreover, a refined study of the ground state allows us to obtain the large deviation rate function explicitly and to deduce that the family of diffusions converges to the set of extreme solutions of the deterministic system. 
\end{abstract}

\maketitle

\section{Introduction and main results}\label{section:Introduction}

In 1890, Peano addressed the existence of solutions to ordinary differential equations (ODEs) driven by continuous but non-Lipschitz functions. Meanwhile, he highlighted that the equation could have several solutions for some initial conditions. Those initial conditions are referred to as Peano's points, and the existence of several solutions is called Peano's phenomenon. However, as we will explain below, Peano's phenomenon disappears when the ODE is perturbed by Gaussian noise, since the resulting stochastic differential equation (SDE) typically has only one solution.

We consider the following stochastic differential equation
\begin{equation}\label{eq:EDE}
X_t^{\varepsilon}= x_0 + \int_{0}^{t} b\left(X_s^{\varepsilon}\right)\text{d}s + \varepsilon W_t, \quad t\in [0,T] \, (T<+\infty),
\end{equation}
where the drift $b:\R^d\rightarrow \R^d$ is a continuous function but not Lipschitz with $b(x_0)=0$, and $W_t = \left(W_t^1, \dots, W_t^d \right)$ is a $d$-dimensional Brownian motion. Under weak conditions on $b$, it is satisfied that Equation  \eqref{eq:EDE} admits an unique strong solution $\left\{X^{\ve}_t \right\}_{t\in [0,T]}$ (see \cite{Fedrizzi_2011}). Intuitively, this is because Brownian motion takes the process instantly away from $x_0$, which is where the unperturbed equation
\begin{equation}\label{eq:ODE}
x(t)= x_0 + \int_0^t b\left(x(s)\right)\text{d}s, \quad t\in [0,T]
\end{equation}
has uniqueness problems. Note that Equation \eqref{eq:ODE} only has uniqueness problems on the set of zeros of $b$. For simplicity, we assume that $b$ has a single zero (or Peano's point) and is the $0$ of $\R^d$. 

As the noise restores uniqueness to Peano's phenomenon, a natural question is to address the limit of the solution of the stochastic equation \eqref{eq:EDE} as the intensity $\ve$ of the noise tends to $0$. Such a procedure is referred to as taking the fluid limit of the family of processes $\left\{ X^\ve \right\}_\ve$. The intuition is that the fluid limit should select some \emph{important} solutions among all the solutions of the original ODE since those are the solutions that are stable under perturbation. Because they are obtained by forcing the dynamics at random, those important solutions should be regarded as the most meaningful from a physical point of view.  

This study builds extensively on previous works, which will be cited as they arise. Nevertheless, we extend the analysis to encompass a broader class of singular potentials and more general SDEs. The main result in this context is due to \cite{Baldi} and \cite{Bafico}. For the one-dimensional case,  they prove that the law of stochastic processes $\left\{X_t^{\ve}\right\}_t$ concentrates on the set of \emph{extremal solutions} \footnote{The solutions that start instantaneously from Peano's points are referred to as \emph{extremal solutions}.} 
$\varphi_1$ and $\varphi_2$ when $\ve$ tends to zero. That is, if $\PP^{\ve}$ is the law of $X^{\ve}$, then $\PP^{\ve}$ converges weakly to $\alpha \delta_{\varphi_1}+(1-\alpha)\delta_{\varphi_2}$, where $\delta_{\varphi_i}$ is the Delta measure of the extremal solution $\varphi_i$, and $\alpha$ is computed explicitly. Another way to study the convergence of $X^\ve$ as $\ve\to 0$  is to study the large deviations. If the drift $b$ is not a Lipschitz function, Equation \eqref{eq:EDE} does not fall within the context of the well-known Freidlin and Wentzell theory, see \cite{F&W}. There are results of large deviations for families of processes that are solutions of Equation \eqref{eq:EDE} in the case where the drift is $b(x)= \left\vert x \right\vert^{\gamma-1}x$ with $\gamma \in (0,1)$ (see \cite{Herrmann} for the case $d=1$ and \cite{Pappalettera} in the case $d>1$). 

In this paper, we generalize the study of large deviations for drifts of the form $b(x)=\nabla U(x),$ where $U:\R^d \to \R$ is a homogeneous function. The idea of taking this kind of drift arose from observing that the homogeneity of $b(x)= \left\vert x \right\vert^{\gamma-1}x$ played a significant role in the work of \cite{Herrmann} and \cite{Pappalettera}. 
Moreover, a fluid limit study is proposed in \cite{delarue2019} under an extensive list of hypotheses for drifts of the form 
\[
b(x)=\nabla U(x), \text{ where } U(0)=0, \text{ and } U(x)= \theta(\frac{x}{|x|})|x|^{1+\gamma} \text{ if } x\neq 0,
\]
being $\theta$ a function defined on the sphere. However, the path we use for the study of the LDP is different and arises from combining the ideas of \cite{Herrmann} and \cite{Pappalettera} with a refinement of the Carmona-Simon exponential bounds for the eigenvectors of Schrödinger operators analyzed in \cite{Carmona78}, \cite{Carmona79}, and \cite{Carmona81}.  
The hypothesis that the drift $b$ comes from a homogeneous potential allows us to establish a relationship between the semigroups 
\[
P_t^{X^\ve}(x)(x)=\E \left[ f(X_t^\ve)|X_0^\ve=x\right] \text{ and } 
T_t^V(f)(x)= \E \left[ f(W_t) e^{-\int_0^t V(W_s)\text{d}s}|W_0=x\right],
\]
for a potential $V$ to be defined in terms of $U$.
By refining Carmona-Simon bounds, we prove 
that the exponential behavior of $X_t^\ve$ depends only on the principal eigenvalue and eigenvector of the linear generator of $T_t^V$. One of our main contributions is a comprehensive analysis of the principal eigenvalue of such an operator in the large deviation rate function.

\bigskip
We will analyze two types of large deviations: a first-order large deviation principle with velocity $\ve^{-2}$ and a second-order large deviation principle with a lower velocity of convergence, depending on the degree of homogeneity of the potential $U$.  For this reason, we recall below the definitions of the large deviation principle and the exponential tightness condition for a rate $\lambda$. 

\begin{definition} Let be $\left(\mathcal{X}, d\right)$ a Polish space and $\left\{ \PP^{\varepsilon} \right\}_{\varepsilon}$ a family of probability measures defined on the $\sigma$-algebra $\mathcal{B}\left(\mathcal{X}\right)$ with $\varepsilon \rightarrow 0$. Let be
$I: \mathcal{X}\rightarrow [0, +\infty]$ a lower semicontinuous function, and $\lambda:\R^{+}\rightarrow \R^{+}$ such that $\lambda(\varepsilon)\rightarrow +\infty$ if $\varepsilon\rightarrow 0$. We say that $\left\{ \PP^{\varepsilon} \right\}_{\varepsilon}$ verify an LDP with \emph{rate function} $I$ and \emph{rate} $\lambda$ if for all open $A \subset \mathcal{X}$,
\[
\liminf_{\varepsilon \rightarrow 0} \lambda(\varepsilon)^{-1} \log \PP^{\varepsilon}(A) \geq - \inf_{x\in A} I(x),
\]
and  for all closed $C \subset \mathcal{X}$,
\[
\limsup_{\varepsilon \rightarrow 0} \lambda(\varepsilon)^{-1} \log \PP^{\varepsilon}(C) \leq - \inf_{x\in C} I(x).
\]
We say that $\left\{\PP^{\ve}\right\}_{\ve}$  is \emph{exponentially tight} with rate $\lambda(\varepsilon)$ if for each $\beta>0$ there exists a compact $K_{\beta} \subseteq \mathcal{X}$ such that 
\[ \limsup_{\ve \to 0} \lambda(\ve)^{-1} \log \PP^{\ve}\left(K_{\beta}^c\right) \leq -\beta.
\]
If $\left( \Omega, \mathcal{A}, \PP \right)$ is a probability space and $\left\{ X^{\varepsilon}\right\}_{\varepsilon}$ are random variables defined on $\left( \Omega, \mathcal{A}, \PP \right)$, we say that $\left\{ X^{\varepsilon}\right\}_{\varepsilon}$ verify an LDP if the induced probability measures defined by $\PP^\ve(A):=\PP(X^\ve \in A)$ verify it.
\end{definition}

\bigskip
We are interested in the case where the random variable $X^{\varepsilon}$ is the strong solution of Equation \eqref{eq:EDE}. Below, we present the two main results of this paper. First, we present the first-order LDP with rate $\ve^{-2}$ for the general case where the drift $b$ has at most linear growth. 
\begin{theorem}[First-order LDP] \label{thm:First LDP}
Let $X^{\ve}=\left\{X_t^{\ve}\right\}_{t\in [0,T]}$ ($0<T<\infty$) be the strong solution of Equation \eqref{eq:EDE} where the drift $b:\R^d \to \R^d$ verifies the following condition,
\[
b(0)=0 \text{ and there exists } A, B\geq 0 \text{ such that }  \left\vert b(x) \right\vert \leq A \left\vert x \right\vert + B \quad \forall x.
\]
Then, the family of stochastic processes $\left\{X^{\ve}\right\}_{\ve}$ verify an LDP on $C_0\left([0,T], \R^d\right)$ when $\ve \to 0$, with rate $\ve^{-2}$ and rate function $I_1: C_0\left([0,T], \R^d\right)\to [0, +\infty]$ such that
\begin{equation}\label{eq:I1}
I_1(\varphi)= \begin{cases} \frac{1}{2}\int_0^T \left\vert\dot{\varphi}(s)-b\left(\varphi(s)\right)\right\vert^2 \text{d}s, & \text{ if } \varphi \in \mathcal{AC}_0\left([0,T], \R^d\right), \\
+\infty, & \text{otherwhise}.
\end{cases}
\end{equation}
\end{theorem}

The proof of Theorem \ref{thm:First LDP} is presented in Section \ref{section:First LDP}. This result shows that an extension of Freilin-Wentzell results is possible when the drift has at most linear growth, even though Equation \eqref{eq:EDE} is not within the hypotheses of their work. As a corollary, we trivially obtain that $ X^\ve$ converges to the set formed by the infinite solutions of the ordinary differential equation \eqref{eq:ODE}. 
Obviously, this result provides little information. For this reason, it is necessary to study large deviations at a slower velocity, which allows us to distinguish the most probable solutions within this set. This is done in the following theorem, where a study of second-order LDP is carried out for the particular case in which the drift comes from a homogeneous potential.

\begin{theorem}[Second-order LDP]\label{thm:Second LDP}
Assume moreover that the drift $b:\R^d \to \R^d$ is such that $b(x)=\nabla U(x)$, being $U:\R^d \to \R$ with $U(0)=0$ and $U(x)=\ti \left\vert x \right\vert^{1+\gamma}$ if $x\neq 0$, where $\gamma \in (0,1)$ and $\theta:D\supset \mathbb{S}^{d-1} \to \R$ is a positive and twice differentiable function in a open subset $D\subset \R^d$. 
Let $\ve_\gamma = \ve^{2\frac{1-\gamma}{1+\gamma}}$ and $\{X^{\ve} \}_{\ve}$ be the family of strong solutions of Equation \eqref{eq:EDE}. Then, $\{X^{\ve} \}_{\ve}$  verify an LDP with rate $\ve_\gamma^{-1}$ and rate function $I_2:C_0\left([0,T], \R^d\right) \to [0, +\infty]$ given by
\begin{equation}\label{eq:I2}
I_2(\varphi)=\begin{cases} \lambda_1 T - \lambda_1\left(T-t_0^\varphi\right)^+, & \text{ if } \varphi \text{ is solution of Equation \eqref{eq:ODE}},\\
+\infty, & \text{ if not.}\end{cases}
\end{equation}
$\lambda_1>0$ is the first eigenvalue (corresponding to the groundstate $\psi_1$) of the Schrödinger operator
\[
-\mathcal{L}(f)=-\frac{1}{2} \Delta f + V.f,
\]
where the potential is the function $V:\R^d \to \R$ such that 
\[
V(x)=\frac{1}{2}\left(\left\vert b(x)\right\vert^2 + {\rm div}\,b(x)\right)=\frac{1}{2}\left(\left\vert \nabla U(x)\right\vert^2 + \Delta U(x)\right).
\]
For a given solution $\varphi$ of Equation \eqref{eq:ODE}, $t_0^\varphi$ denotes the time of exit from $0$, i.e. $t_0^\varphi= \sup \{ t\geq 0: \, \varphi(t)=0\}$.
\end{theorem}

This is the main contribution of the paper. The proof of Theorem \ref{thm:Second LDP} is presented in Section \ref{section:Second LDP}. For the proof, we first analyze the exponential behavior of the density function $p^\ve(t,x)$ of the random variable $X_t^\ve$ for a fixed $t$, conditioned to $X_0^\ve = x$. Then, we observe that this density can be written in terms of the integral kernel of the Schrödinger semigroup 
\[
T_t(f)(x)=\E \left[f(W_t)e^{-\int_0^t V(W_s)\text{d}s}|W_0=x\right],
\]
whose infinitesimal generator is the operator $-\mathcal{L}(f)=-\frac{1}{2}\Delta f + V.f$. We prove that the exponential behavior of $p^\ve(t,x)$ only depends on $U$ and $\psi_1$, the ground state of $-\mathcal{L}$. Then, from refining the Carmona-Simon bounds for $\psi_1$ in our particular case, we succeed in proving that if $\ve_\gamma= \ve^{2 \frac{1-\gamma}{1+\gamma}}$, then the limit $\underset{\ve \to 0}{\lim} \ve_\gamma \log\left(p^\ve(t,x)\right)$ exists and it is $-\lambda_1 t-g(x)$, being $\lambda_1$ the first eigenvalue of $-\mathcal{L}$, and $g$ the only homogeneous solution of the partial differential equation 
\[
\left\langle \nabla U(x), \nabla g(x)\right\rangle = -\lambda_1.
\]
The existence and uniqueness of such a function $g$ will be deduced from the proof of the theorem. Finally, from the study of the exponential behavior of $p^\ve(t,x)$, we derive an LDP for the family of stochastic processes $\left\{X^\ve \right\}_\ve$, and prove that the rate function $I_2(\varphi)$ only depends on the final time $T$, being it $-\lambda_1 T -g\left(\varphi(T)\right)$ if $\varphi$ is a solution of Equation \eqref{eq:ODE}. Moreover, we deduce that, in this case, $g\left(\varphi(t)\right)=-\lambda_1  \left(t-t_0^\varphi\right)^+$ for all $0\leq t\leq T$.  As a corollary, we deduce that $X^\ve$ converges to the set of extremal solutions of Equation \eqref{eq:ODE}; i.e., the fluid limit is not a single trajectory (as occurs in a classical fluid limit) but a set of trajectories.

\bigskip

Finally, in Section \ref{section:final comments}, we briefly present the work of \cite{Herrmann} and \cite{Pappalettera}, and offer concluding remarks on the possibility of extending their work and the results presented in this article. There, we also comment on the difficulty we encountered in solving this problem, arising from the study of the convergence of the nonlinear semigroups associated with the stochastic processes $X^\ve$, due to the impossibility of proving the uniqueness of viscosity solutions for the Hamilton-Jacobi equations involved.

\bigskip
{\bf Notation comment:} for typing convenience, we write the Euclidean norm in $\R^d$ as $\left\vert . \right\vert$; $f(x)\approx g(x)$ means that $\underset{x \to .}{\lim} \frac{f(x)}{g(x)}=1$, and $f(x) \lesssim g(x)$ means that $0<\underset{x\to .}{\lim} \frac{f(x)}{g(x)} \leq 1$. $C_0\left([0,T], \R^d\right)$ and $ \mathcal{AC}_0\left([0,T], \R^d\right)$ denote respectively the space of functions $\varphi:[0,T] \rightarrow \R$ continuous and absolutely continuous such that $\varphi(0)=0$. $L^1_{loc}(\R^d)$ refers to the space of locally integrable functions $f:\R^d \to \R$, and $f\in L^p(\R^d)$ if moreover $\int \left\vert f(x) \right\vert^p \text{d}x < \infty$.  


\section{Proof of Theorem \ref{thm:First LDP} (First-order LDP)} \label{section:First LDP}

In this section, we present a proof of Theorem \ref{thm:First LDP}, which consists of a generalization of the work of Freidlin-Wentzell for the general case of drifts that exhibit at most linear growth. This result is not original, as it is based on the following lemma, presented as Theorem {\bf 2.14} in Pappalettera's Master's Thesis (see \cite{Pappalettera}) for the general case of continuous and unbounded drift functions $b:\R^d \to \R^d$. Furthermore, in \cite{Pappalettera2022}, this result is generalized for stochastic differential equations in infinite-dimensional Hilbert spaces perturbed by a cylindrical Wiener process.  

\begin{lemma}[Theorem 2.14 from \cite{Pappalettera}]\label{thm:Pappalettera}
Consider the family $\left\{ X^\ve \right\}_\ve$ of solutions of Equation \eqref{eq:EDE} where the drift $b: \R^d \to \R^d$ is a continuous function. If
\begin{enumerate}
\item[($H_1$)] $\left\{ X^\ve \right\}_\ve$ is exponentially tight with rate $\ve^{-2}$,
\item[($H_2$)] $\forall \ve>0$, the Doléans exponential $\mathcal{E}\left(-\frac{1}{\ve} \int_0^. b(X_s^\ve) \text{d}W_s \right)$ is a martingale on $[0,T]$,
\end{enumerate}
then $\left\{ X^\ve \right\}_\ve$ verify an LDP with rate $\ve^{-2}$ and rate function $I_1$ defined on \eqref{eq:I1}.
\end{lemma}

The proof of this theorem is presented in detail in sections 2.3 and 2.4 of \cite{Pappalettera}. For completeness, we include an idea of the proof below. 

\begin{proof}
\emph{Step 1.} If $b$ is continuous and bounded, then $b$ is approximated by a sequence of Lipschitz functions $\left\{ b_n\right\} _n$, where Freidlin-Wentzell results are valid for each $\left\{ X^{n, \ve} \right\}_\ve$, the corresponding solution of Equation \eqref{eq:EDE} for $b_n$ instead of $b$. It is proved that the Freidlin-Wentzell action functional is still valid for $\left\{ X^\ve \right\}_\ve$. \emph{Step 2.} If $b$ is unbounded and continuous, then it is approximated by a sequence of continuous and bounded functions $\left\{ b_R \right\}_R$. Then, by Step 1, an LDP is valid for each $\left\{ X^{R, \ve} \right\}_\ve$, the corresponding solution of Equation \eqref{eq:EDE} for $b_R$ instead of $b$. Then it is proved that if ($H_2$) is verified, then the Freidlin-Wentzell condition (FW1) is verified for the rate $\ve^{-2}$ and $I_1$. Moreover, if the exponential tightness condition ($H_1$) holds, then $I_1$ has compact level sets, and the Varadhan condition (V2) holds. Finally, an LDP is verified with rates $\ve^{-2}$ and $I_1$.
\end{proof}
\bigskip

Assume that the drift $b:\R^d \to \R^d$ verifies the following condition,
\begin{condition}\label{Condition1 for b}
$b(0)=0$ and there exists $A, B\geq 0$ such that  $\left\vert b(x) \right\vert \leq A \left\vert x \right\vert + B \quad \forall x$.
\end{condition}
\bigskip

\begin{proof}[Proof of Theorem \ref{thm:First LDP}]
Due to the previous lemma, it suffices to prove that if the drift $b:\R^d \to \R^d$ verifies Condition \ref{Condition1 for b}, then $\left\{ X^\ve \right\}_\ve$ verifies conditions ($H_1$) (exponential tightness) and ($H_2$) (martingale property).
\bigskip

\emph{Exponential tightness condition:} Let $\beta>0$ be fixed, we want to construct a compact set $K_\beta \subseteq C_0\left([0,T], \R^d \right)$ such that $\underset{\ve \to 0}{\limsup} \, \ve^2 \log \PP\left( X^\ve \notin K_\beta\right) \leq -\beta$. Let us define for $\delta_1>0$ and $\delta_2>0$ to be chosen later, the following set:
\[
K_{\delta_1,\delta_2}=\left\{ f\in C_{0, \alpha}([0,T],\R^d): \, \left\Vert f\right\Vert_{\infty} \leq \delta_1; \,
\sup_{0\leq s <t\leq T} \frac{\left \vert f(t)-f(s)\right\vert}{(t-s)^\alpha}\leq \delta_2 \right\},
\]
being $C_{0,\alpha} \left([0,T], \R^d \right)$ the space of functions $f=(f_1, \dots, f_d):[0,T]\rightarrow \R^d$ such that each component $f_i$ is Hölder continuous of index $\alpha$, equipped with the norm
\[
\left\Vert f \right\Vert_{C_{0, \alpha}([0,T], \R^d)}:=\left\Vert f \right\Vert_{\infty} + \underset{0\leq s < t \leq T}{\sup} \frac{\left\vert f(t)-f(s)\right\vert}{(t-s)^{\alpha}}.
\]
Endowed with this norm, $C_{0, \alpha}([0,T], \R^d)$ is a Banach space. Note that the functions belonging to $K_{\delta_1, \delta_2}$ are equicontinuous and totally bounded, then by the Arzelá-Ascoli Theorem, the set $K_{\delta_1, \delta_2}$ is compact in the topology of the uniform convergence. So, it is enough to choose $\delta_1>0$ and $\delta_2>0$ such that $K_{\beta}:= K_{\delta_1, \delta_2}$ verifies $\underset{\ve \to 0}{\limsup} \, \ve^2 \log \PP\left( X^\ve \notin K_\beta\right) \leq -\beta$. Due to Condition \ref{Condition1 for b}, $X_t^{\varepsilon}$ verifies 
$\left\vert X_t^{\varepsilon} \right\vert \leq A \int_0^T \left\vert X_s^{\varepsilon}\right\vert \text{d}s + BT + \underset{s\leq T}{\sup} \left\vert \varepsilon W_s \right\vert,$
and 
$
\left\Vert X^\ve \right\Vert_{\infty}=\underset{s\leq T}{\sup} \left\vert X_s^{\varepsilon} \right\vert \leq \left( BT + \underset{s\leq T}{\sup} \left\vert \varepsilon W_s \right\vert \right)e^{AT},
$
by Gronwall lemma.

If 
$
\delta_1 < \left\Vert X^{\varepsilon}\right\Vert_{\infty} \leq ( BT + \underset{s\leq T}{\sup}\left\vert \ve W_s \right\vert) e^{AT},
$
then
$
\underset{s\leq T}{\sup} \left\vert \ve W_s \right\vert > \delta_1 e^{-AT}-BT,
$
and
\begin{gather*}
\PP\left( \left\Vert X^\ve \right\Vert_\infty > \delta_1 \right) \leq
\PP\left( \sup_{s\leq T} \left\vert \ve W_s \right\vert > \delta_1 e^{-AT}-BT  \right) \leq \ve \frac{4}{\sqrt{2\pi}} \frac{d^{\frac{3}{2}} \sqrt{T}}{M}e^{-\frac{1}{2T}\frac{M^2}{\ve^2d}},
\end{gather*}
if we use the known bound for the Brownian motion
\begin{gather*}
\PP \left(\sup_{s\leq T} \left\vert \varepsilon W_s\right\vert >M \right) \leq \varepsilon \frac{1}{\sqrt{2\pi}} \frac{d^{\frac{3}{2}} \sqrt{T}}{M}e^{-\frac{1}{2T}\frac{M^2}{\varepsilon^2d}} 
\end{gather*}
for $M= \delta_1 e^{-AT}-BT>0$.

Moreover,
$
\underset{0\leq s<t \leq T}{\sup}\left\vert X_t^\ve - X_s^\ve \right\vert \leq \left(B(t-s) + \underset{0\leq s<t \leq T}{\sup} \ve \left\vert W_t - W_s \right\vert \right)e^{AT},
$ 
and
\[
\frac{\left\vert X_t^\ve - X_s^\ve \right\vert}{(t-s)^{\alpha}} \leq C_1 \left((t-s)^{1-\alpha} + \ve \frac{\left\vert W_t-W_s\right\vert}{(t-s)^{\alpha}}\right)
\]
for some constant $C_1>0$. This inequality implies two things. First, that $X^\ve$ belongs to $C_{0,\alpha}([0,T], \R^d)$ if $\alpha\in (0,\frac{1}{2})$. In fact, for almost every $\omega$, we have
\[
\frac{\left\vert X_t^\ve - X_s^\ve \right\vert}{(t-s)^{\alpha}} \leq C_1 \left(T^{1-\alpha} + \ve \sup_{0\leq s<t \leq T} \left\{ \frac{\left\vert W_t-W_s\right\vert}{(t-s)^{\frac{1}{2}} \log^{\frac{1}{2}}\left(\frac{1}{t-s}\right)} (t-s)^{\frac{1}{2}-\alpha} \log^{\frac{1}{2}}\left(\frac{1}{t-s}\right)\right\}\right),
\]
which is bounded due to the modulus of continuity of Brownian motion. Moreover, for almost every $\omega$ the function $X^\ve$ is bounded in $[0,T]$. If
\[
\delta_2 < \sup_{0\leq s<t\leq T} \frac{\left\vert X_t^\ve - X_s^\ve \right\vert}{(t-s)^{\alpha}} \leq C_1 \left(T^{1-\alpha} + \left\Vert \ve W \right\Vert_{C_{0,\alpha}([0,T],\R^d)}\right),
\]
then $\left\Vert \ve W \right\Vert_{C_{0,\alpha}([0,T],\R^d)}>\frac{\delta_2}{C_1}-T^{1-\alpha}$. We choose $\delta_2$ sufficiently large such that $\frac{\delta_2}{C_1}-T^{1-\alpha}>0$. Since $W$ is a Gaussian random variable taking values on the separable Banach space $C_{0,\alpha}([0, T], \R^d),$ there exists a $t_0>0$ such that for every $t<t_0$ we have $\mathbb{E}\left[ e^{t \left\Vert W \right\Vert^2_{C_{0,\alpha}([0, T], \R^d)}}\right]<\infty$ (see Theorem {\bf 6.5} from \cite{Gine}). Then, for one of those $t<t_0$, we have
\begin{align*}
\PP \left( \sup_{0\leq s<t\leq T} \frac{\left\vert X_t^\ve-X_s^\ve \right\vert }{(t-s)^\alpha}>\delta_2\right) & \leq 
\PP \left( \left\Vert W \right\Vert_{C_{0,\alpha}([0,T], \R^d)} > \frac{\frac{\delta_2}{C_1}-T^{1-\alpha}}{\ve}\right)\\
& = \PP \left( e^{t  \left\Vert W \right\Vert^2_{C_{0,\alpha}([0,T], \R^d)}} > e^{\frac{t}{\ve^2}\left(\frac{\delta_2}{C_1}-T^{1-\alpha}\right)^2}\right)\\
& \leq C_2 e^{-\frac{t}{\ve^2}\left(\frac{\delta_2}{C_1}-T^{1-\alpha}\right)^2},
\end{align*}
with $C_2=\mathbb{E}\left[ e^{t  \left\Vert W \right\Vert^2_{C_{0,\alpha}([0,T], \R^d)}}\right] < \infty$.

Finally, 
\begin{align*}
\PP\left(X^\ve \notin K_{\delta_1,\delta_2}\right) & \leq 
\PP \left( \left\Vert X^\ve \right\Vert_\infty >\delta_1 \right) + \PP\left( \sup_{0\leq s<t\leq T} \frac{\left\vert X_t^\ve-X_s^\ve \right\vert }{(t-s)^\alpha}>\delta_2\right)\\
& \leq \ve \frac{1}{\sqrt{2\pi}} \frac{d^{\frac{3}{2}} \sqrt{T}}{M}e^{-\frac{1}{2T}\frac{M^2}{\ve^2d}} + C_2 e^{-\frac{t}{\ve^2}\left(\frac{\delta_2}{C_1}-T^{1-\alpha}\right)^2},
\end{align*}
and
\[
\limsup_{\ve \to 0} \ve^2 \log \PP\left(X^\ve \notin K_{\delta_1,\delta_2}\right) \leq \max \left\{ -\frac{1}{2T}\frac{M^2}{d}, -t (\frac{\delta_2}{C_1}-T^{1-\alpha})^2 \right\} = -\beta
\]
if we choose $\delta_1 = \left( \sqrt{2T\beta}d + BT \right)e^{AT}$ and $\delta_2=C_1 \left( \sqrt{\frac{\beta}{T}}+ T^{1-\alpha}\right)$.
\bigskip

\emph{Martingale condition:} It is enough to prove that for each $\ve>0$ the Novikov's condition holds, that is,
\[
\E \left[ e^{\frac{1}{2 \ve^2} \int_0^T \left\vert b(X_s^\ve)\right\vert^2 \text{d}s}\right]<\infty.
\]
Again, Condition \eqref{Condition1 for b} and Gronwall's lemma imply that $\underset{s\leq T}{\sup} \left\vert X_s^\ve \right\vert \leq \left(BT + \underset{s\leq T}{\sup} \left\vert \ve W_s \right\vert \right)e^{AT}$. Then,
\[
\left\vert b(X_s^\ve) \right\vert^2  \leq \left(A \left\vert X_s^\ve\right\vert + B \right)^2 
 \leq 2 \left(A^2 \left(\sup_{s\leq T}\left\vert X_s^\ve \right\vert \right)^2 + B^2 \right)
 := C_1 + \ve^2 C_2 \, \underset{s\leq T}{\sup} \left\vert  W_s \right\vert^2,
\]
and 
\begin{align*}
\E \left[ e^{\frac{1}{2 \ve^2} \int_0^T \left\vert b(X_s^\ve)\right\vert^2 \text{d}s} \right] 
&  \leq \E \left[ e^{\frac{T}{2\ve^2} \left(  C_1 + \ve^2 C_2 \, \underset{s\leq T}{\sup} \left\vert  W_s \right\vert^2 \right)} \right]
 = e^{\frac{C_1 T}{2\ve^2}}  \E \left[ e^{\frac{C_2 T}{2} \, \underset{s\leq T}{\sup} \left\vert  W_s \right\vert^2 } \right].
\end{align*}
For each $\ve$, $e^{\frac{C_1 T}{2\ve^2}}<\infty$ and, since $T<\infty$, the reflection principle and exponential integrability of Gaussian random variables imply that the last mean is finite:
\begin{align*}
\E \left[ e^{\frac{C_2 T}{2} \, \underset{s\leq T}{\sup} \left\vert  W_s \right\vert^2 } \right] & = \E \left[ e^{\frac{C_2 T}{2} \, \underset{s\leq T}{\sup} \left\vert  W_s \right\vert^2 } {\bf 1}_{\{ \underset{s\leq T}{\sup} \left\vert W_s \right\vert > a \}}\right] +
\E \left[ e^{\frac{C_2 T}{2} \, \underset{s\leq T}{\sup} \left\vert  W_s \right\vert^2 } {\bf 1}_{\{ \underset{s\leq T}{\sup} \left\vert W_s \right\vert \leq a \}}\right]\\
& \leq \E \left[ e^{\frac{C_2 T}{2} \, \underset{s\leq T}{\sup} \left\vert  W_s \right\vert^2 } \right] 2 \PP \left( \left\vert W_T \right\vert >a \right) + e^{\frac{C_2 T}{2} a^2},
\end{align*}
then, if we choose $a>0$ sufficiently large so that $1-2 \PP \left( \left\vert W_T \right\vert >a \right)>0$, we have
\[
\E \left[ e^{\frac{C_2 T}{2} \, \underset{s\leq T}{\sup} \left\vert  W_s \right\vert^2 } \right] \leq \frac{e^{\frac{C_2 T}{2} a^2}}{1-2 \PP \left( \left\vert W_T \right\vert >a \right)} < \infty.
\]
\end{proof}

\section{Proof of Theorem \ref{thm:Second LDP} (Second-order LDP)} \label{section:Second LDP}

In this section, we present the proof of Theorem \ref{thm:Second LDP}, which consists of a second-order LDP study for the particular case where the drift $b$, in addition to having at most linear growth, comes from a homogeneous potential $U$. As we mentioned before, the reason for considering this kind of drift comes from observing that the homogeneity of $b(x)=|x|^{\gamma -1}x$ played a significant role in the work of \cite{Herrmann} and \cite{Pappalettera}. A peculiarity of these functions is that their derivatives preserve the homogeneity property, which we will use to study large deviations. 
\bigskip

Let be $U:\R^d\to\R$ such that $U(0)=0$ and $U(x)=\ti |x|^{\gamma+1}$ if $x\neq 0$, where $\gamma\in(0,1)$ and $\theta:D \supset\mathbb{S}^{d-1}\to \R$ is a positive and twice differentiable function in an open $D\subset \R^d$. Let us consider the drift 
\begin{equation} \label{eq:b extended}
b(x)=\nabla U(x)=|x|^{\gamma}\left[\nabla\ti+ \left((1+\gamma)\ti -\left\langle \nabla \ti, \frac{x}{|x|} \right\rangle \right) \frac {x}{|x|} \right],
\end{equation}
if $x\neq 0$ and $b(0)=0$. Observe that $b$ is a homogeneous function of degree $\gamma$ that is continuous but non-Lipschitz on $\R^d$, however it is Lipschitz on any compact set on $\R^d\smallsetminus \{0\}$. As we will explain in more detail in the proof, we require that the function $\theta$ be positive to ensure that the solutions to Equation \eqref{eq:ODE} behave expansively and cannot return to the origin once they leave it. Let be $\theta_1:D \supset\mathbb{S}^{d-1}\to \R^d$ such that $b(x)=\theta_1(\frac{x}{|x|})|x|^\gamma$. Note that $\theta_1(\frac{x}{|x|})$ can be decomposed into a radial and a tangential component, given by $(1+\gamma)\theta(\frac{x}{|x|}) \frac{x}{|x|}$ and $\nabla \theta(\frac{x}{|x|})-\left\langle \nabla \theta(\frac{x}{|x|}), \frac{x}{|x|}\right\rangle \frac{x}{|x|}$ respectively. 
\bigskip

From Theorem \ref{thm:First LDP}, we know that a first-order LDP (with rate $\ve^{-2}$) is verified since $\left\vert b(x)\right\vert \leq a_\infty \left(\left\vert x\right\vert +1\right)$, and the growth condition is verified with $A=B=a_\infty = \underset{z\in \mathbb{S}^{d-1}}{\sup} |\theta_1(z)| < \infty$. In this section, we prove Theorem \ref{thm:Second LDP}; i.e. we prove that if $X^\ve$ is the strong solution of Equation \eqref{eq:EDE}, being $b$ the function defined in Equation \eqref{eq:b extended}, then $\left\{ X^\ve \right\}_\ve$ verify an LDP as $\ve \to 0$ with rate $\ve_\gamma^{-1}$, being $\ve_\gamma:=\ve^{2\frac{1-\gamma}{1+\gamma}}$, and rate function defined on Equation \eqref{eq:I2}. For the proof, we study first the exponential behavior of $p^\ve(t,x)$, the density function of the random variable $X_t^\ve$ for a fixed time $t\in [0,T]$, and conditioned to $X_0^\ve=x$. 
\bigskip

The proof of Theorem \ref{thm:Second LDP} is organized as follows.
In subsection \ref{subsection:Exponential behavior of the density}, we prove that the density $p^\ve(t,x)$ verifies
\[
p^\ve(t,x) = \frac{1}{\ve^d \ve_\gamma^{\frac{d}{2}}} e^{U\left( \xe\right)} \sum_{j=1}^{\infty} e^{-\lambda_j \frac{t}{\ve_\gamma}}\psi_j(0) \psi_j\left(\xe\right), 
\]
being $\lambda_j$ and $\psi_j$, respectively, the eigenvalues and normalized eigenfunctions of the Schrödinger operator 
$-\mathcal L(f)(x):=-\frac12\Delta (f)(x)+V(x)f(x)$. Moreover, we prove that the only term that matters at an exponential level for $p^\ve(t,x)$ is the one corresponding to the first eigenvector $\psi_1$ (the ground state of the Schrödinger operator), that is
\[
\lim_{\ve\to0}\ve_\gamma\log(p^\ve(t,x)) =
-\lambda_1 t +
\lim_{\ve\to0}\ve_\gamma\log \left[ e^{U\left( \xe\right) }\psi_1\left(\xe\right)\right].
\]
Possibly, this is the main contribution of the paper: to show why only $\psi_1$ (and $\lambda_1$) end up influencing the second-order large deviation rate function, and calculating this last limit for deducing an LDP for the rate $\ve_\gamma^{-1}$.

To prove this result and to compute the last limit, we make a refinement of the techniques proposed by Carmona-Simon to bound each function $e^U(x)\psi_j(x)$ when $|x| \to \infty$. To do this, we first need to verify that the potential $V$ satisfies the hypotheses of the Carmona-Simon results. This is donne in subsection \ref{subsection:Decomposition of the potential}, where we prove that the potential $V$ can be decomposed as $V=V_1-V_2$ such that $V_1$ is bounded below and $V_1\in L_{loc}^1(\R^d)$, and $V_2\geq 0$ with $V_2\in L^p(\R^d)$ for a certain $p>\frac{d}{2}$.

In subsection \ref{subsection:Upper bound for the density p}, we find an upper bound for the normalized (in the $L^2(\R^d)$ norm) eigenfunctions $\psi_j$ using Carmona-Simon techniques. Then, from this bound we prove in Lemma \ref{lemma:uniformly bounded exp(u)psij} a bound for each term $e^{U(x)}\psi_j(x)$ when $|x|\to \infty$, and we get the upper bound for the limit 
\[
\underset{\ve\to0}{\lim}\ve_\gamma\log \left[ e^{U\left( \xe\right) }\psi_1\left(\xe\right)\right] \leq -g(x)
\] 
in Lemma \ref{proposition: upper bound for p}, being $g$ a homogeneous function of degree $1-\gamma$.

In subsection \ref{Lower bound for the density p}, we get a lower bound for the previous limit, which coincides with the upper bound if, moreover, the function $g$ is a solution of the partial differential equation $\left\langle \nabla U(x), \nabla g(x) \right\rangle =-\lambda_1$.

Subsection \ref{subsection: About the existence of g} is devoted to discussing the existence and uniqueness of a homogeneous function $g$ that verifies the equation $\left\langle \nabla U(x), \nabla g(x)\right\rangle=-\lambda_1$.

Finally, in subsection \ref{subsection: Second order LDP}, we prove Theorem \ref{thm:Second LDP} from the lower and upper bounds obtained for the density $p^\ve(t,x)$.


\subsection{Exponential behavior of the density} \label{subsection:Exponential behavior of the density}
In this subsection, we present different representations for the density function of the random variable $X^\ve_t$ to study its exponential behavior when $\ve$ tends to $0$. 
\bigskip
 
\begin{lemma}\label{Prop:density p(t,x)}
Let be $p^\ve(t,x)$ the density of the r.v. $X^\ve_t$ (conditioning to have $X^\ve_0=x$) for a fixed time $t$. Then, 
\begin{equation} \label{eq:p(t,x)}
p^\ve(t,x)=\frac1{\ve^d(2\pi t)^{\frac d2}} e^{ U\left(\xe\right)-\frac{|x|^2}{2t\ve^2}} \E\left[e^{\int_0^{\frac t{\ve_\gamma}}V\left(W_s\right)ds} \big| W_{\frac{t}{\ve_\gamma}}=\frac x{\ve\ve_\gamma^{1/2}}\right],
\end{equation}
where 
\[
V(x)=\frac{1}{2}\left(|b(x)|^2 + {\rm{div}}\,b(x)\right)=\frac{1}{2}\left(|\nabla U(x)|^2+\Delta U(x)\right).
\]
Moreover, $p^\ve(t,x)$ can be written as
\begin{align} \label{eq: densidad Mercer}
p^\ve(t,x) & = \frac{1}{\ve^d \ve_\gamma^{\frac{d}{2}}} e^{U\left( \xe\right)} \sum_{j=1}^{\infty} e^{-\lambda_j \frac{t}{\ve_\gamma}}\psi_j(0) \psi_j\left(\xe\right), 
\end{align}
being $-\infty<\lambda_1 < \lambda_2 \leq \lambda_3\leq \dots$ and $\psi_j$ respectively the eigenvalues and normalized eigenfunctions (in the $L^2(\R^d)$ norm) of the Schrödinger operator 
$-\mathcal L(f)(x):=-\frac12\Delta (f)(x)+V(x)f(x)$.
\end{lemma}

\begin{proof}
If $f:\R^d \to \R$ is an arbitrary function, then  $\E\left(f(X_t^\ve)\right)= \int_{\R^d}f(x)p^\ve(t,x) \text{d}x$, and Equation \eqref{eq:p(t,x)} is obtained by analyzing the expectation $\E\left(f(X_t^\ve)\right)$. The proof follows the same scheme as the proof of Corollary 1 in \cite{Herrmann} and Proposition 3.5 in \cite{Pappalettera}. To give completeness to this article, we include the proof below.

Define $Z^\ve_t=\frac1\ve X^\ve_t$. Since $b$ is homogenous of degree $\gamma$, this process satisfies the SDE
\[
\begin{cases}
dZ^\ve_t=\frac1\ve b(\ve Z^\ve_t)\text{d}t+\text{d}W_t=\ve^{\gamma-1}b(Z^\ve_t)\text{d}t+\text{d}W_t,\\
Z^\ve_0=0.
\end{cases}
\]
Let's introduce its semigroup  $P^{Z^\ve}_tf(x)=\E[f(x+Z^{\ve}_t)]$. As we proved before, Novikov's condition is verified, and the Cameron-Martin-Girsanov formula allows writing
 \[
P^{Z^\ve}_t(f)(x)=\E[e^{\ve^{\gamma-1} \int_0^t b(x+ W_s)\text{d}W_s-\frac{\ve^{2(\gamma-1)}}2\int_0^t|b(x+ W_s)|^2\text{d}s}f(x+ W_t)].
\]
The next thing to do is to eliminate the stochastic integral in the above expectation. Using  Itô's Formula and that $b(x)=\nabla U(x)$ we have
\[
U(x+ W_t)-U(x)=\int_0^t b(x+ W_s)\text{d}W_s+\frac{1}2\int_0^t{\rm{div}}\,b(x+ W_s)\text{d}s,
\]
and, since $U(0)=0$, we obtain
\begin{align*}
P^{Z^\ve}_tf(x)&= \E[f(x+Z^\ve_t)]\\
& =e^{-(\ve^{\gamma-1})U(x)}
\E[e^{-\frac{1}2\int_0^t((\ve^{(\gamma-1)}|b(x+W_s)|)^2+\ve^{\gamma-1} {\rm{ div }}\,b(x+ W_s))\text{d}s}e^{(\ve^{\gamma-1})U(x+ W_t)}f(x+ W_t)].
\end{align*}
Our $Z^\ve$ starts from zero, then our interest is in
\begin{align*}
\E[f(X^\ve_t)]& =\E[f(\ve Z^\ve_t)]=P^{Z^\ve}_tf_\ve(0)\\
& =\E[e^{-\frac{1}2\int_0^t((\ve^{(\gamma-1)}|b(W_s)|)^2+\ve^{\gamma-1} {\rm{ div }}\,b(W_s))\text{d}s}e^{(\ve^{\gamma-1})U(W_t)}f(\ve W_t)],
\end{align*}
where $f_\ve(x):=f(\ve x)$. Thus
\begin{align*}
\E[f(X^\ve_t)] &= \int_{\R^d}f(\ve y)p_t(y)e^{(\ve^{(\gamma-1)})U(y)}\E[e^{-\frac{1}2\int_0^t((\ve^{(\gamma-1)}|b(W_s|)^2+\ve^{\gamma-1} {\rm{ div }}\,b(W_s))\text{d}s} |W_t=y]\text{d}y\\
&=\frac1{\ve^d}\int_{\R^d}f(y)p_t(\frac{y}{\ve})e^{(\ve^{(\gamma-1)})U(\frac y{\ve})}\E[e^{-\frac{1}2\int_0^t((\ve^{(\gamma-1)}|b(W_s)|)^2+\ve^{\gamma-1} {\rm{ div }}\,b(W_s))\text{d}s}|W_t=\frac{y}\ve]\text{d}y,
\end{align*}
where $p_t(y)$ is the density function of the Gaussian random variable $W_t$ for fixed $t$. Let's simplify this expression. In the first place, we have 
$
e^{(\ve^{(\gamma-1)})U(\frac y\ve)}=e^{\frac{U(y)}{\ve^2}}
$.
In the second place, by using the invariance of the scale of the Brownian motion, we have $\ve^{\frac12}_{\gamma}W_{\frac{\cdot}{\ve_\gamma}}\stackrel{d}=W_{\cdot}.$ Then,

\begin{multline*}
\E\left[e^{-\frac{1}2\int_0^t((\ve^{(\gamma-1)}|b(W_s)|)^2+\ve^{\gamma-1} {\rm{ div }}\,b(W_s))\text{d}s}|W_t=\frac{y}\ve\right]\\
=\E\left[ e^{-\frac{1}2\int_0^t((\ve^{(\gamma-1)}(\ve_\gamma)^{\frac\gamma2}|b(W_{\frac{s}{\ve_{\gamma}}})|)^2+\ve^{\gamma-1}\ve^{\frac{\gamma-1}2}_{\gamma} {\rm{ div }}\,b(W_{\frac{s}{\ve_{\gamma}}})\text{d}s}|W_{\frac{t}{\ve_{\gamma}}}=\frac{y}{\ve_{\gamma}^{\frac12}\ve}\right]\\
= \E\left[ e^{{ \displaystyle-\frac{1}2\int_0^t\left(|b(W_{\frac{s}{\ve_{\gamma}}})|^2+{\rm{ div }}\,b(W_{\frac{s}{\ve_{\gamma}}})\right)\frac{\text{d}s}{\ve_\gamma}}}\Big|W_{\frac{t}{\ve_{\gamma}}}=\frac{y}{\ve_{\gamma}^{\frac12}\ve}\right],
\end{multline*}
where $\ve_\gamma= \ve^{2\frac{1-\gamma}{1+\gamma}}$.
For obtaining the last equality we have used that $b(\xi x)=\xi^{\gamma}b(x)$ and also ${\rm{div}}\, b(\xi x)=\xi^{\gamma-1}{\rm{div}}\, b(x)$ if $\xi>0$.  Now, we can make the change of variable $s=s\ve_\gamma$ into the integral in the exponential; then the above expression is equal to
\[
\E\left[e^{-\frac{1}2\int_0^{\frac{t}{\ve_\gamma}}\left(|b(W_s)|^2+{\rm{ div }}\,b(W_s)\right)\text{d}s}| W_{
\frac{t}{\ve_{\gamma}}}=\frac{y}{\ve_{\gamma}^{\frac12}\ve}\right].
\]
Finally, we can write 
\begin{align*}
\E[f(X^\ve_t)] & = \int_{\R^d}f(y)p^\ve(t,y)
\text{d}y\\
& =\frac1{\ve^d}\int_{\R^d}f(y)p_t(\frac{y}{\ve})e^{\frac{U(y)}{\ve^2}}\E\left[e^{-\frac{1}2\int_0^{\frac{t}{\ve_\gamma}}\left(|b(W_s)|^2+{\rm{ div }}\,b(W_s)\right)\text{d}s}\Big|W_{\frac{t}{\ve_{\gamma}}}=\frac{y}{\ve_{\gamma}^{\frac12}\ve}\right] \text{d}y.
\end{align*}
Hence, if define $V(x)=\frac{1}{2}\left(|b(x)|^2+ {\rm div}\, b(x)\right)$ and noting that $U\left(\frac{y}{\ve \ve_\gamma^{\frac{1}{2}}}\right)=\frac{U(y)}{\ve^2}$, we have the desired result:
\[
p^\ve(t,y)=\frac1{\ve^d}p_t(\frac{y}{\ve})e^{\frac{U(y)}{\ve^2}}\E\left[e^{-\frac{1}2\int_0^{\frac{t}{\ve_\gamma}}V(W_s)\text{d}s}\Big|W_{\frac{t}{\ve_{\gamma}}}=\frac{y}{\ve_{\gamma}^{\frac12}\ve}\right].
\]

Now, consider the semigroup
\[ T_t(f)(x)=\E\left[f(W_t)e^{-\int_0^tV(W_s)\text{d}s} \big| W_0=x\right],\] whose infinitesimal generator is the Schrödinger operator $-\mathcal L$. This is an unbounded self-adjoint operator acting in a dense subspace of $L^2(\R^d)$ and as in our case $V(x)\to\infty$ when $|x|\to\infty$ we know by \cite{Carmona78}, \cite{Carmona79} and \cite{Carmona81} that it has a discrete spectrum with eigenvalues $-\infty < \lambda_1 < \lambda_2 \leq \lambda_ 3\leq \dots$ and respectively normalized eigenfunctions $\{\psi_i\}_{i=1}^\infty$ (i.e. $-\mathcal{L}(\psi_i)=\lambda_i \psi_i$ with $\left\Vert \psi_i \right\Vert_{L^2(\R^d)}=1$ for all $i\geq 1$). By the Mercer theorem we get that the integral kernel\footnote{i.e. $a(t,x)$ is such that $ T_t(f)(x)= \int_{R^d} f(y) a_t(x,y)\text{d}y$} $a_t(t,y)$ of $T_t$, 
\[
a_t(x,y)= \frac{1}{(2\pi t)^\frac{d}{2}} e^{-\frac{|x-y|^2}{2t}} \E \left[ e^{-\frac{1}{2} \int_0^t V(W_s)\text{d}s} \big| W_0=x, \, W_t=y \right],
\]
can be written as 
$
a_t(x,y)= \sum_{j=1}^{\infty} e^{-\lambda_j t} \psi_j(x)\psi_j(y). 
$
Then, the density of $X_t^\ve$ can be written as
\begin{align*} 
p^\ve(t,x) & =\frac{1}{\ve^d \ve_{\gamma}^{\frac{d}{2}}} e^{U \left(\xe\right)} a_{\frac{t}{\ve_\gamma}}\left(0,\xe\right)
  = \frac{1}{\ve^d \ve_\gamma^{\frac{d}{2}}} e^{U\left( \xe\right)} \sum_{j=1}^{\infty} e^{-\lambda_j \frac{t}{\ve_\gamma}}\psi_j(0) \psi_j\left(\xe\right). 
\end{align*}
\end{proof}

Moreover, the first eigenvalue $\lambda_1$ in the representation \eqref{eq: densidad Mercer} is positive.

\begin{lemma}
Let $\lambda_1$ be the first eigenvalue corresponding to the Schrödinger operator $-\mathcal{L}(f)=-\frac{1}{2} \Delta f + V.f$, then $\lambda_1 >0$.
\end{lemma}
\begin{proof}
Let $\psi\in Dom\left(-\mathcal{L}\right)$. Note that
\begin{align*}
\left\langle \psi, -\mathcal{L}(\psi)\right\rangle_{L^2(\R^d)} & = \int \psi(x) \left( -\frac{1}{2}\Delta \psi(x) + V(x) \psi(x) \right)\text{d}x \\
& =\frac{1}{2} \left(\int \left\vert \nabla \psi(x)\right\vert^2 \text{d}x + \int \left(\left\vert \nabla U(x)\right\vert^2 + \Delta U(x)\right)\left(\psi(x)\right)^2 \text{d}x\right)\\
& = \frac{1}{2} \int \left\vert \nabla \psi(x)-\psi(x)\nabla U(x)\right\vert^2 \text{d}x \geq 0,
\end{align*}
then it must be $\lambda_1 \geq 0$. Let's see that $\lambda_1$ cannot be $ 0$. Suppose that there exists $\psi_1 \in Dom\left(-\mathcal{L}\right)$ such that $-\mathcal{L}(\psi_1)=0$, then
\[
0=\left\langle \psi_1, -\mathcal{L}(\psi_1)\right\rangle_{L^2(\R^d)}=\frac{1}{2}\int \left\vert \nabla \psi_1(x)-\psi_1(x)\nabla U(x)\right\vert^2 \text{d}x,
\]
and it must be $\nabla \psi_1(x)-\psi_1(x)\nabla U(x)=0$ for a.e. $x$. Then $\nabla \left(\log \psi_1\right)(x)=\nabla U(x)$, and there exits $C\in \R$ such that $\psi_1(x)=ce^{U(x)}$ for a.e. $x$, but this is not possible since $e^{U} \notin Dom\left(-\mathcal{L}\right)$.
\end{proof}

In the following, we use this representation of the density to study its exponential behavior when $\ve$ tends to zero. We first prove that the only term that matters in Equation \eqref{eq: densidad Mercer} corresponds to the first eigenvector $\psi_1$ (the ground state of the Schrödinger operator). Moreover, it is pointed out in \cite{Carmona79} that a simple consequence of Feynman-Kac's formula is that $\psi_1$ can be chosen everywhere positive and locally bounded away from zero. Therefore, the logarithm appearing in the following limit is well-defined.

\begin{lemma}\label{prop:p converges to expU psi1} Let $\lambda_1 > 0$ and $\psi_1$ be the first eigenvalue and the ground state of the Schrödinger operator $-\mathcal{L}(f)(x)=-\frac{1}{2}\Delta f(x)+ V(x)f(x)$, then
\begin{gather*}
\lim_{\ve\to0}\ve_\gamma\log(p^\ve(t,x)) =
-\lambda_1 t +
\lim_{\ve\to0}\ve_\gamma\log \left[ e^{U\left( \xe\right) }\psi_1\left(\xe\right)\right].
\end{gather*}
\end{lemma}

\begin{proof}
From Equation \eqref{eq: densidad Mercer}, we have
\begin{align*}
\lim_{\ve\to 0}\ve_\gamma\log(p^\ve(t,x)) 
&=\lim_{\ve\to0}\ve_\gamma\log \left[\sum_{j=1}^\infty e^{-\lambda_j\frac{t}{\ve_\gamma}} e^{U\left( \xe\right)}\psi_j(0)\psi_j\left(\xe\right)\right]\\
&= \lim_{\ve\to0}\ve_\gamma\log \left[e^{-\lambda_1 \frac{t}{\ve_\gamma}}\sum_{j=1}^\infty e^{-(\lambda_j-\lambda_1)\frac{t}{\ve_\gamma}} e^{U\left( \xe\right)}\psi_j(0)\psi_j\left(\xe\right)\right]\\
&= -\lambda_1 t + \lim_{\ve\to0}\ve_\gamma\log \Big[ e^{U\left( \xe\right)}\psi_1(0)\psi_1\left(\xe\right)  + S_\ve(x) \Big],
\end{align*}
being ${\displaystyle S_\ve(x)=\sum_{j=2}^\infty e^{-(\lambda_j-\lambda_1)\frac{t}{\ve_\gamma}} e^{U\left( \xe\right)}\psi_j(0)\psi_j\left(\xe\right)}$. It suffices to prove that $S_\ve(x)\to 0$  as $\ve \to 0$.
We will see in Lemma \ref{lemma:uniformly bounded exp(u)psij} 
that there exist constants $C,p, r>0$ such that $e^{U(x)} \left\vert \psi_j(x) \right\vert \leq C \left(1+\lambda_j\right)^{\frac{p}{2}}$ uniformly in $j$, if $\left \vert x \right\vert \geq r$. Moreover, as $0<\lambda_1 < \lambda_2 \leq \lambda_3 \leq \dots$ with $\lambda_j \to \infty$ if $j \to \infty$, there exists $k_0$ such that $\lambda_{k_0} > \lambda_1 + \lambda_2$. Then,

\begin{align*}
\left\vert S_\ve (x) \right\vert & \leq \sum_{j=2}^{k_0 -1 } e^{-(\lambda_j-\lambda_1)\frac{t}{\ve_\gamma}} e^{U\left( \xe\right)}\left\vert \psi_j(0)\right\vert \left\vert \psi_j\left(\xe\right)\right\vert \\
& \quad + e^{-\lambda_2 \frac{t}{\ve_\gamma}} \sum_{j=k_0}^{\infty } e^{-(\lambda_j-\lambda_1-\lambda_2)\frac{t}{\ve_\gamma}} e^{U\left( \xe\right)}\left\vert \psi_j(0)\right\vert \left\vert \psi_j\left(\xe\right)\right\vert \\
& \leq k_0 C e^{-(\lambda_2 - \lambda_1) \frac{t}{\ve_\gamma}} \left(1 + \lambda_{k_0 -1}\right)^{\frac{p}{2}} \sup_{2 \leq j \leq k_0 -1} \left\vert \psi_j(0)\right\vert \\
& \quad + C e^{-\lambda_2 \frac{t}{\ve_\gamma}} \sum_{j=k_0}^{\infty } e^{-(\lambda_j-\lambda_1-\lambda_2)\frac{t}{\ve_\gamma}}  \left(1+\lambda_j\right)^{\frac{p}{2}}\left\vert\psi_j(0)\right\vert  \\
& \to 0 \text{ if } \ve \to 0,
\end{align*}
since the series is convergent ($\left\vert\psi_j(0)\right\vert$ also grows at most polynomially with $\lambda_j$).
\end{proof}

In the following subsections, we use a refinement of the techniques proposed by Carmona-Simon to prove that:
\begin{enumerate}
\item There exist positive constants $C,p$ such that $e^{U(x)} \left\vert \psi_j(x) \right\vert \leq C \left(1+\lambda_j\right)^{\frac{p}{2}}$ when $\left\vert x \right\vert \to \infty$ (which we used in the previous proof), see Lemma \ref{lemma:uniformly bounded exp(u)psij}.
\item The following limit exists
\[
\lim_{\ve\to0}\ve_\gamma\log \left[ e^{U\left( \xe\right)}\psi_1\left(\xe\right)\right]= -g(x),
\]
being $g:\R^d \to \R$ such that $g(0)=0$, and $\left\langle \nabla U(x), \nabla g(x)\right\rangle=-\lambda_1$ (the uniqueness of $g$ will be deduced in subsection \ref{subsection: About the existence of g}). In lemmas \ref{proposition: upper bound for p} and \ref{proposition:lower bound for p}, we prove respectively the upper and lower bound for the above limit.
\end{enumerate}

\subsection{Decomposition of the potential}\label{subsection:Decomposition of the potential}
In this subsection, we prove that the potential $V$ is in the context of Carmona-Simon work.
\bigskip

\begin{lemma} \label{lemma:Decomposition of the potential}
The potential 
$
V(x)=\frac{1}{2}\left(|b(x)|^2 + {\rm{div}}\,b(x)\right)=\frac{1}{2}\left(|\nabla U(x)|^2+\Delta U(x)\right),
$
can be decomposed into $ V=V_1-V_2$ such that $V_1$ is bounded below and $V_1 \in L^1_{loc}(\R^d)$ on one hand, and $V_2 \geq 0$ and $V_2\in L^p(\R^d)$ for a certain $p>\frac d2$ on the other.
\end{lemma}

\begin{proof} Since $U$ is a homogeneous function of degree $\gamma+ 1$, then $\nabla U$ is homogeneous of degree $\gamma$ and $\Delta U$ is homogeneous of degree $\gamma - 1$. Therefore, there exist functions $\theta_1: \mathbb{S}^{d-1}\rightarrow \R^d$ and $\theta_2: \mathbb{S}^{d-1} \rightarrow \R$ such that:
\[ 
V(x)=\frac12\left( |\theta_1(\frac x{|x|})|^2 |x|^{2\gamma}+\theta_2(\frac x{|x|})|x|^{\gamma-1} \right).
\]
First, we decompose $\theta_2$ as $\theta_2= \theta_{2+}-\theta_{2-}$. Let  $z$ be a real positive to be chosen later, and let define $V_2(x)=\frac{1}{2} \theta_{2-}(\frac{x}{\left\vert x \right\vert}) |x|^{\gamma-1} \mathbf{1}_{\{ |x|\leq z \}}$ and $V_1(x)=V(x)+V_2(x)$. Then,
\[
V_1(x) = \frac12\left[ |\theta_1(\frac x{|x|})|^2|x|^{2\gamma}\mathbf 1_{\{|x|\leq z\}}+\theta_{2+}(\frac x{|x|})|x|^{\gamma-1}
  |x|^{2\gamma} \left(|\theta_1(\frac x{|x|})|^2 -\theta_{2-}(\frac x{|x|})|x|^{-\gamma-1}\right)\mathbf 1_{\{|x|>z\}}\right]
\]
The only term in the above sum that can be negative  is the last one, but we have when $|x|>z$
$$
|\theta_1(\frac x{|x|})|^2 -\theta_{2-}(\frac x{|x|})|x|^{-\gamma-1} \ge |\theta_1(\frac x{|x|})|^2 -
\underset{y\in \mathbb{S}^{d-1}}{\sup} |\theta_{2-}(y)|(\frac1z)^{1+\gamma}\ge0,
$$ 
wherever we take $z$ such that
  $z^{1+\gamma}\ge\frac{\underset{y\in \mathbb{S}^{d-1}}{\sup} |\theta_{2-}(y)|}{\left\vert\theta_1(\frac{x}{|x|})\right\vert^2}$ (we will prove in a moment that $\left\vert\theta_1(y)\right\vert^2>0$ for all $y\in \mathbb{S}^{d-1}$). But we have that this inequality holds if
  $$z=\left(\frac{\underset{y\in \mathbb{S}^{d-1}}{\sup} |\theta_{2-}(y)|}{\underset{y\in \mathbb{S}^{d-1}}{\inf}\left\vert\theta_1(y)\right\vert^2}\right)^{\frac1{1+\gamma}},$$
which we will take in what follows as the limit of validity of our result.  
Note that $z=z(\theta)$. In this form, we have $V_1\ge0$.
For the other term in the decomposition, we have
  $$\int_{\R^d}|V_2(x)|^p\text{d}x\le \frac{1}{2^p}\left\Vert \theta_{2}\right\Vert_\infty^p\sigma_d(\mathbb{S}^{d-1})\int_0^{z}\frac1{r^{p(1-\gamma)-d+1}}\text{d}r.$$ 
This last integral is convergent whenever $p(1-\gamma)-d<0$, thus when $p<\frac{d}{1-\gamma}$. Since for $0<\gamma<1$ we have $1-\gamma<2$ we get $\frac d2<\frac d{1-\gamma}$,  implying that we can always chose a $p$ such that $\frac d2<p<\frac d{1-\gamma}$. We will choose one of these exponents and remark that, in fact, $p:=p(\gamma)$.
\end{proof}   

\begin{lemma}
If define $\theta_1$ such that $\nabla U(x)=\theta_1(\frac{x}{|x|})\left\vert x \right\vert^{\gamma},$ then $\left\vert \theta_1(y)\right\vert^2 >0$ for all $y \in \mathbb{S}^{d-1}$.
\end{lemma}   
\begin{proof} Since
\[
\frac{\partial}{\partial x_i}U(x) = \left\vert x \right\vert ^{\gamma} \left[ -\left\langle \nabla \theta(\frac{x}{|x|}), \frac{x}{|x|} \right\rangle \frac{x_i}{|x|}+ \left(\frac{\partial}{\partial x_i}\theta\right)(\frac{x}{|x|}) + \theta(\frac{x}{|x|}) (1+\gamma) \frac{x_i}{|x|}\right],
\]
then $\theta_1(y)=\theta(y)(1+\gamma)y+\nabla \theta(y)-\left\langle \nabla \theta (y), y\right\rangle y$ and
\[
\left\vert \theta_1(y)\right\vert^2= (\theta(y))^2 (1+\gamma)^2 + \left\vert \nabla \theta (y)\right\vert^2 \sin^2(\alpha)>0
\]
for all $y\in \mathbb{S}^{d-1}$, where $\alpha$ is the angle formed by $\nabla \theta(y)$ and $y$.
\end{proof}

\begin{remark}
Moreover, note that $\left\vert \theta_1(y) \right\vert ^2> (\theta(y))^2>0$. This remark will be used later.
\end{remark}

\subsection{Upper bound for the density $p^\ve(t,x)$}\label{subsection:Upper bound for the density p}
In this subsection, we use Carmona-Simon techniques to get an upper bound for the eigenvectors $\psi_j$ (particularly for the ground state $\psi_1$) when $\left\vert x \right\vert \to \infty$. We do this because we must refine the Carmona-Simon bounds for our particular case. If we use its bounds ($\psi_1(x) \leq D(\delta)e^{-\delta |x|^{\gamma + 1}}$, see \cite{Carmona78}), then the limit $\underset{\ve \to 0}{\lim} \ve_\gamma \log \left[ e^{U(\xe)} \psi_1(\xe) \right]$ is bounded by an expression that explodes.
\bigskip

If $\psi_j$ is a normalized eigenvector of the  Schrödinger operator $-\mathcal{L}$ with eigenvalue $\lambda_j$, then
$$T_t(\psi_j)=e^{-\lambda_j t} \psi_j, \text{ being } T_t(f)(x)=\E_x \left[ f(W_t) e^{-\int_0^t V(W_s)\text{d}s} \right],$$
with  $V(x)=V_1(x)-V_2(x)$ decomposed as in the previous section. Then,
\begin{align*}
\left\vert \psi_j(x)\right\vert^2 & = e^{2\lambda_j t} \left( \E_x\left[ \psi_j(W_t) e^{-\int_0^t V(W_s)\text{d}s}\right]\right)^2\\
& \leq \left\Vert \psi_j \right\Vert_{\infty}^2 e^{2\lambda_j t}  \E_x\left[ e^{-2\int_0^t V_1(W_s)\text{d}s}\right] \E_x\left[ e^{2\int_0^t V_2(W_s)\text{d}s}\right]
\end{align*}
Let $a>0$ be a parameter to be determined later and $V_1^a(x)=\underset{y\in \overline{B(x,a)}}{\inf}V_1(y)$. Moreover, let be $\hat{V}_1(x)=2V_1(x)=\left[ \left\vert \nabla U(x)\right\vert^2 + \left( \theta_{2+}(\frac{x}{|x|})-\theta_{2-}(\frac{x}{|x|}) {\bf 1}_{\{|x|>z\}} \right)\left\vert x\right\vert^{\gamma-1}\right]$ and $\hat{V}_2(x)=2V_2(x)=\theta_{2-}(\frac{x}{|x|}){\bf 1}_{\{ |x|\leq z\}}\left\vert x\right\vert^{\gamma - 1}$ (analogously we define $\hat{V}_1^a$). Then,
\begin{align*}
\E_x\left[ e^{-2\int_0^t V_1(W_s)\text{d}s}\right] & = \E_x\left[ e^{-\int_0^t\hat{V}_1(W_s)\text{d}s}
{\bf 1}_{\{\underset{s\leq t}{\sup}\left\vert W_s-x\right\vert\leq a\}}\right] +  \E_x\left[ e^{-\int_0^t\hat{V}_1(W_s)\text{d}s}{\bf 1}_{\{
\underset{s\leq t}{\sup}\left\vert W_s-x\right\vert >  a\}}\right]\\
& \leq e^{-t\hat{V}_1^a(x)}+ \PP_x\left(\underset{s\leq t}{\sup}\left\vert W_s-x\right\vert > a\right),
\end{align*}
where
\[
\PP_x\left(\underset{s\leq t}{\sup}\left\vert W_s-x\right\vert > a\right) \leq 2d \frac{1}{(2\pi)^{\frac{d}{2}}} \int_{\frac{a}{\sqrt{t}}}^{+\infty} r^{d-1}e^{-\frac{r^2}{2}}dr  \leq c_d \left[ \left(\frac{a}{\sqrt{t}}\right)^{\frac{d}{2}}+1 \right]e^{-\frac{a^2}{2t}},
\]
and $c_d>0$ only depends on $d$. On the other hand, using Equation {\bf 2.2} from \cite{Carmona78} for $r=1$, we have that for all $t\geq 0$,
\[
\E_x\left[ e^{2\int_0^t V_2(W_s)ds}\right] \leq \underset{x\in \R^d}{\sup} \E_x\left[ e^{\int_0^t \hat{V}_2(W_s)ds}\right] \leq K_{\eta}e^{c(p)^{\frac{1}{\eta}} \left\Vert \hat{V}_2 \right\Vert_{p}^{\frac{1}{\eta}}t},
\]
being $\eta=1-\frac{d}{2p}$, $c(p)=\frac{1}{(2\pi)^{\frac{d}{2p}}}\left(1-\frac{1}{p}\right)^{\left(1-\frac{1}{p}\right)^{\frac{d}{2}}}$, and $K_\eta$ some positive constant that depends only on $\eta$. Finally,
\begin{align}\label{eq:cota psij}
\left\vert \psi_j(x)\right\vert^2 & \leq \left\Vert \psi_j \right\Vert_{\infty}^2 e^{2\lambda_j t} K_{\eta}e^{c(p)^{\frac{1}{\eta}} \left\Vert \hat{V}_2 \right\Vert_{p}^{\frac{1}{\eta}}t} \times \left[e^{-t\hat{V}_1^a(x)}+  c_d \left[ \left(\frac{a}{\sqrt{t}}\right)^{\frac{d}{2}}+1 \right]e^{-\frac{a^2}{2t}}\right]\\
&= K_\eta \left\Vert \psi_j \right\Vert_{\infty}^2 e^{\left( 2\lambda_j +c(p)^{\frac{1}{\eta}} \left\Vert \hat{V}_2 \right\Vert_{p}^{\frac{1}{\eta}}\right)t}  \left[e^{-t\hat{V}_1^a(x)}+  c_d \left[ \left(\frac{a}{\sqrt{t}}\right)^{\frac{d}{2}}+1 \right]e^{-\frac{a^2}{2t}}\right],
\end{align}
for all $t>0$ and $a>0$. 
\bigskip

Now we want to choose $t>0$ and $a>0$ appropriately so that we can prove the bound for each $\left\vert e^{U(x)}\psi_j(x)\right\vert$ and also get an upper bound for the limit 
\[ \lim_{\ve\to0}\ve_\gamma\log \Big[ \exp\left(U\left( \xe\right)\right)\psi_1\left(\xe\right)\Big].
\]
Those results are presented as lemmas \ref{lemma:uniformly bounded exp(u)psij} and \ref{proposition: upper bound for p}.
\bigskip

\begin{lemma}\label{lemma:uniformly bounded exp(u)psij} 
There exists constants $C,r>0$ such that $\left\vert e^{U(x)}\psi_j(x)\right\vert \leq C \left(1+\lambda_j\right)^{\frac{p}{2}}$, for each $j$ when $|x|>r$, being $p>\frac{d}{2}$ the constant defined in Lemma \ref{lemma:Decomposition of the potential}.
\end{lemma}
\begin{proof}
Let be $c_j=2\lambda_j +c(p)^{\frac{1}{\eta}} \left\Vert \hat{V}_2 \right\Vert_{p}^{\frac{1}{\eta}}$ and take $t=\frac{ma}{\sqrt{\hat{V}_1^a(x)}}$ in Equation \eqref{eq:cota psij} with $m>0$. Then, if $|x|\to \infty$,
\begin{align*}
\left\vert \psi_j (x)\right\vert^2 & \lesssim K_\eta \left\Vert \psi_j \right\Vert_{\infty}^2 \left[ e^{(c_j -\hat{V}_1^a(x))\frac{ma}{\sqrt{\hat{V}_1^a(x)}}} + c_d \left(\frac{a}{m}\sqrt{\hat{V}_1^a(x)}\right)^{\frac{d}{4}}e^{-\frac{\sqrt{\hat{V}_1^a(x)}a}{2m}}\right]\\
& \approx K_\eta \left\Vert \psi_j \right\Vert_{\infty}^2 \left[ e^{-ma\sqrt{\hat{V}_1^a (x)}} + c_d \left(\frac{a}{m}\sqrt{\hat{V}_1^a(x)}\right)^{\frac{d}{4}}e^{-\frac{\sqrt{\hat{V}_1^a(x)}a}{2m}}\right].
\end{align*}
Let be $\delta \in (0,1)$ and $r>0$ such that $\left(\frac{a}{m}\sqrt{\hat{V}_1^a(x)}\right)^{\frac{d}{4}}e^{-\frac{\sqrt{\hat{V}_1^a(x)}a}{2m}} < e^{-\delta \frac{\sqrt{\hat{V}_1^a(x)}a}{2m}}$ if $|x|\geq r$. Then,
\[
\left\vert \psi_j (x)\right\vert^2 
\lesssim K_\eta \left\Vert \psi_j \right\Vert_{\infty}^2 e^{-\inf\left\{ ma \sqrt{\hat{V}_1^a(x)},\frac{\delta a \sqrt{\hat{V}_1^a (x)}}{2m} \right\} },
\]
for all $m>0$ and $a>0$. Now, we have
\[
e^{U(x)}\left\vert \psi_j(x)\right\vert \lesssim \sqrt{K_\eta} \left\Vert \psi_j \right\Vert_{\infty} e^{U(x)-\inf\{\frac{m}{2}, \frac{\delta}{4m} \} a \sqrt{\hat{V}_1^a(x)}}.
\]
First, we choose $a>0$ and $m>0$ so that the exponent in the last equation is negative when $|x| \to \infty$. Then, we prove that there exists a positive constant $C$ independent of $j$ such that $\sqrt{K_\eta} \left\Vert \psi_j \right\Vert_{\infty} \leq C\left( 1+ \lambda_j \right)^{\frac{p}{2}}$.

For the first part, note that if $|x|\to \infty$, then $\hat{V}_1^a(x) \approx \left\vert \nabla U(x)\right\vert^2 = | \theta_1(\frac{x}{|x|})|^2 |x|^{2\gamma}$. Then, we take 
\[ 
a(x)=\max \{ \frac{2}{m}, \frac{4m}{\delta}\} \underset{y\in \mathbb{S}^{d-1}}{\sup}\frac{\theta(y)}{\left\vert \theta_1(y)\right\vert} \left\vert x \right\vert < \max \{ \frac{2}{m}, \frac{4m}{\delta}\} \left\vert x \right\vert,
\]
for a fixed $m>0$, and we get that effectively $U(x)-\inf\{\frac{m}{2}, \frac{\delta}{4m} \} a \sqrt{\hat{V}_1^a(x)} <0$ if $|x|$ is sufficiently large.

Finally, it is enough to prove that there exist a constant $C$ such that $\left\Vert \psi_j \right\Vert_{\infty} \leq C\left(1+\lambda_j\right)^{\frac{p}{2}}$ (with abuse of notation, we will then take $C=\sqrt{K_\eta}C$). Since $\psi_j$ verifies $-\Delta \psi_j = 2 \left(\lambda_j - V\right) \psi_j$, and $V(x)\to +\infty$ when $|x|\to \infty$, standard elliptic regularity gives that for every $k\in \N$ there exists a constant $C_k>0$ such that
\[
\left\Vert \psi_j \right\Vert_{H^{k+2}} \leq C_k \left(1+ \lambda_j\right) \left\Vert \psi_j \right\Vert_{H^k},
\]
where $\left\Vert .\right\Vert_{H^k}$ denotes the usual norm on the Sobolev space on $\R^d$, $H^k=\{ f\in L^2(\R^d): \partial^{\alpha}f \in L^2(\R^d) \; \forall |\alpha|\leq k \}$. Iterating, we have $\left\Vert \psi_j \right\Vert_{H^s} \leq C_s \left(1+ \lambda_j\right)^{\frac{s}{2}}$ for all $s=2k$. Moreover, if $s>\frac{d}{2}$ (that is why we take $s=p$ from Lemma \ref{lemma:Decomposition of the potential}), then Sobolev embedding gives $\left\Vert \psi_j \right\Vert_\infty \leq C \left\Vert \psi_j \right\Vert_{H^s}$, and combining with the previous bound, we get $\left\Vert \psi_j \right\Vert_{\infty} \leq C C_p \left(1+\lambda_j \right)^{\frac{p}{2}}=C\left(1+\lambda_j \right)^{\frac{p}{2}}$, with abuse of notation on $C$.
\end{proof}

Now, we are able to prove that if $g:\R^d\to \R$ is a homogeneous function of degree $1-\gamma$ verifying $\left\langle \nabla U(x), \nabla g(x)\right\rangle = -\lambda_1$, then we can get an upper bound for the limit $\underset{\ve \to 0}{\lim}\ve_\gamma\log(p^\ve(t,x))$ as a function of $g$. But, before presenting this result, we introduce some observations about the behavior of the ground state $\psi_1$ and present the heuristic that allowed us to arrive at the candidate function $g(x)$.

\begin{remark}
From Carmona-Simon work, we know that there exists a function $\rho$ such that\\
$\lim_{|x|\to\infty}-\frac{\log(\psi_1(x))}{\rho(x)}=1$. If $V(x)\ge0$, the function $\rho$ is the Agmon's distance 
\[ \rho(x)=\inf\big\{\int_0^1\sqrt{2V(\gamma(s))}|\dot\gamma(s))|ds:\,\gamma\in\mathcal C,\gamma:[0,1]\to\R^d, \gamma(0)=0,\gamma(1)=x\big\},\]
(see \cite{Agmon}). Since, in our case, $V$ need not be positive, we will approximate $\rho$ from the following heuristic. We conjecture that $\log\psi_1(x)\approx -\rho(x)\approx -U(x)-g(x),$ when $|x|\to\infty$, being $g$ a homogeneous function of degree $1-\gamma$ such that $\left\langle \nabla U(x), \nabla g(x)\right\rangle=-\lambda_1$, since if we define $\psi(x)=e^{-U(x)-g(x)}$, then
\[
-\mathcal{L}(\psi)(x) - \lambda_1 \psi(x)  = \psi(x)     \left[ -\lambda_1 - \left\langle \nabla U(x), \nabla g(x)\right\rangle + \Delta U(x)-\frac{1}{2}\left\vert \nabla g(x)\right\vert ^2 - \frac{1}{2} \Delta g(x)\right].
\]
Due to the homogeneity of $U$ and $g$, the term $\Delta U(x)-\frac{1}{2}\left\vert \nabla g(x)\right\vert ^2 - \frac{1}{2} \Delta g(x) \to 0$ if $|x|\to \infty$, and we choose $g$ such that $-\lambda_1 - \left\langle \nabla U(x), \nabla g(x)\right\rangle=0$. In the particular case where $\theta = \frac{1}{1+\gamma}$,  the solution is $g(x)= -\lambda_1 \frac{|x|^{1-\gamma}}{1-\gamma}$, and agrees with the results in \cite{Pappalettera}. 
\end{remark}
\bigskip

\begin{lemma}[Upper bound for the density]\label{proposition: upper bound for p}
Let $g:\R^d\to \R$ be a homogeneous function of degree $1-\gamma$ verifying $\left\langle \nabla U(x), \nabla g(x)\right\rangle = -\lambda_1$. Then,
\[
\lim_{\ve\to0}\ve_\gamma\log(p^\ve (t,x))\leq -\lambda_1 t -g(x), \, \forall x, \, \forall t.
\]
\end{lemma}

\begin{proof}
Due to the homogeneity of $g$, we can define a function $\theta_g: \mathbb{S}^{d-1}\to \R$ such that $g(x)=\theta_g(\frac{x}{|x|}) \left\vert x\right\vert^{1-\gamma}$. From the proof of Lemma \ref{lemma:uniformly bounded exp(u)psij} for $j=1$, we know that
\[
0< \psi_1(x) \lesssim \sqrt{K_\eta} \left\Vert \psi_1 \right\Vert_{\infty} e^{-\inf\{\frac{m}{2}, \frac{\delta}{4m} \} a \sqrt{\hat{V}_1^a(x)}}, \quad \forall m>0, \, \forall a>0.
\] 
Now, we want to choose $a=a(x)$ such that $\psi_1(x)\lesssim e^{-U(x)-g(x)}$ when $|x|\to \infty$. If $|x|\to \infty$, then 
$
\inf\{\frac{m}{2}, \frac{\delta}{4m} \} a \sqrt{\hat{V}_1^a(x)} \approx U(x)+g(x) 
$
if
\[ 
a(x) = \frac{1}{\inf \{\frac{m}{2}, \frac{\delta}{4m} \}} \left[ \frac{\ti}{\left\vert \theta_1(\frac{x}{|x|})\right\vert}  + \frac{ \theta_g(\frac{x}{|x|})}{\left\vert \theta_1(\frac{x}{|x|})\right\vert} |x|^{-2\gamma}\right]|x|.
\]
Then, with this choice of $a(x)$, we have that
\begin{align*}
\lim_{\ve\to0}\ve_\gamma\log(p^\ve (t,x))& =
-\lambda_1 t+\lim_{\ve\to0}\ve_\gamma\log \Big[ e^{U\left( \xe\right)}\psi_1\left(\xe\right)\Big] \\
& \leq -\lambda_1 t + \lim_{\ve\to0}\ve_\gamma\log \Big[ e^{U\left( \xe\right)} \sqrt{K_\eta} \left\Vert \psi_1 \right\Vert_{\infty} e^{-U\left( \xe\right)-g\left( \xe\right)}\Big]\\
& =  -\lambda_1 t + \lim_{\ve\to0}\ve_\gamma g\left( \xe\right)\\
&= -\lambda_1 t -g(x),
\end{align*}
due to the homogeneity of $g$.
\end{proof}

\subsection{Lower bound for the density $p^\ve(t,x)$}\label{Lower bound for the density p}

For the lower bound, we use Lemma {\bf 4.1} from \cite{Carmona78}, which is presented below.

\begin{lemma}[Lemma {\bf 4.1} from \cite{Carmona78}]\label{lemma:lower bound for psi1}
For each $x\in \R^d$, $x\neq 0$, and for each positive real numbers $\alpha_j, b_j, a_j, $ and $t$ such that
$a_j^2 >t$, $\alpha_j < \frac{a_j}{2}$, $l\left( [-a_j, a_j] \cap [-x_j-b_j, -x_j+b_j]\right)>\alpha_j,$ where $l({\bf I})$ denotes the length of the interval ${\bf I}$, the following lower bound for the ground state is verified
\begin{align*}
-\log \psi_1(x) & \leq -\lambda_1 t - \log\left(\eta(b)\right)
+ d\log \left(2t \sqrt{2\pi t} \right) -\sum_j \log \left(\alpha_j^2 a_j \right)\\
& \qquad + \frac{9}{8t} \sum_j a_j^2 + t \sup \{V_1(y): \, |y_j-x_j|<a_j \},  
\end{align*}
being $\eta(b):= \inf\{ \psi_1(y):\, |y_j|\leq b_j \, \forall j \}$.
\end{lemma} 
\begin{proof}
See Lemma {\bf 4.1} from \cite{Carmona78}.
\end{proof}
\bigskip

\begin{lemma}[Lower bound for the density]\label{proposition:lower bound for p}
Let $g:\R^d\to \R$ be a function verifying $\left\langle \nabla U(x), \nabla g(x)\right\rangle = -\lambda_1$. Then,
\[
\lim_{\ve\to0}\ve_\gamma\log(p^\ve (t,x))\geq -\lambda_1 t -g(x), \, \forall x, \, \forall t.
\]
\end{lemma}

\begin{proof}
We need a lower bound for
\[
\ve_\gamma \left[ U\left(\xe \right) + \log \left(\psi_1\left(\xe\right)\right) \right].
\]
Since we do not have a favorite direction, we apply Lemma \ref{lemma:lower bound for psi1} to the positive parameters $a$, $b$, $\alpha$, and $t$, which will be chosen appropriately to get the lower bound as a function of $g$. By Lemma \ref{lemma:lower bound for psi1}, we have
\begin{align*}
\ve_\gamma \left[ U(\xe) + \log(\psi_1(\xe)) \right] & \geq 
\ve_\gamma \Big[ U(\xe) - \sup\left\{ V_1(y):\, \left\vert y_j-\frac{x_j}{\ve \ve_\gamma^{\frac{1}{2}}}\right\vert <a \, \forall j \right\} t \\
& \quad + \lambda_1t  -d\log(2t\sqrt{2\pi t})\\
& \quad + \log(\eta(b)) + d \left(\log(\alpha^2 a)-\frac{9}{8t}a^2\right) \Big]\\
& := L_{\ve,x}.
\end{align*}
Then, taking the following parameters $\alpha=b>0$ constants, $t=t_{\ve,x}= \ve_\gamma^{-1} \frac{U(x)}{\left\vert \nabla U(x)\right\vert^2}$, 
and $a=a_{t,x}$ such that
$
a_{t,x}^2= \frac{8}{9d}\ve_{\gamma}^{-1}t_{\ve,x}\left[g(x)+ \lambda_1 \frac{U(x)}{\left\vert \nabla U(x) \right\vert^2}+ \frac{1}{2}\frac{\ve_\gamma}{\ve^2} U(x)\right], 
$
we obtain that $\underset{\ve \to 0}{\lim} L_{\ve,x}=-\underset{\ve \to 0}{\lim} \ve_\gamma  g(\xe)=-g(x)$, and the proof is concluded since
\begin{align*}
\lim_{\ve\to0}\ve_\gamma\log(p^\ve (t,x))& =
-\lambda_1 t+\lim_{\ve\to0}\ve_\gamma\log \Big[ e^{U\left( \xe\right)}\psi_1\left(\xe\right)\Big] \\
& \geq -\lambda_1 t + \lim_{\ve\to0} L_{\ve,x}\\
&= -\lambda_1 t -g(x).
\end{align*}
\end{proof}

\subsection{About the existence of a homogeneous function $g$ such that $\left\langle \nabla U(x), \nabla g(x)\right\rangle = -\lambda_1$.}\label{subsection: About the existence of g}

From the upper and lower bounds for the density (see lemmas \ref{proposition: upper bound for p} and \ref{proposition:lower bound for p}), we get that if there exists a homogeneous function of degree $1-\gamma$ such that $\left\langle \nabla U(x), \nabla g(x)\right\rangle = -\lambda_1$, then the limit 
\[
\lim_{\ve \to 0} \ve_\gamma \log \left(p^\ve (t,x)\right),
\]
exists and it is $-\lambda_1 t -g(x)$. Then, we can deduce that if such a function $g$ exists, it must be unique. In this subsection, we include some comments on the existence of a homogeneous function $g$ that verifies the equation $\left\langle \nabla U(x), \nabla g(x)\right\rangle=-\lambda_1$.
\bigskip

First, note that due to the homogeneity of $\nabla U(x)$, if there exists a homogeneous function $g$ verifying $\left\langle \nabla U(x), \nabla g(x)\right\rangle=-\lambda_1$, then it must be homogeneous of degree $1-\gamma$. Moreover, note that if $\varphi$ is solution for Equation \eqref{eq:ODE}, then
\[
\frac{\partial}{\partial t}\left(g(\varphi(t))\right)= \left\langle \nabla g(\varphi(t)), \dot{\varphi}(t)\right\rangle =
\left\langle \nabla g(\varphi(t)), \nabla U(\varphi(t)) \right\rangle=
 -\lambda_1,
\]
i.e., $\varphi$ is a characteristic curve for the PDE $\left\langle \nabla U(x), \nabla g(x)\right\rangle =-\lambda_1$. Then, we can define $g$ along the characteristics by $g(\varphi(t))=g(0)-\lambda_1 (t-t_0^\varphi)^+,$ being $t_0^\varphi=\inf\{ t\geq 0: \varphi(t)=0\}$. In addition, the characteristic curves for this PDE are well behaved in the sense that they cannot cross each other since the system \eqref{eq:ODE} has a uniqueness of the flow on $\R^d\smallsetminus \{0\}$. That is, although our system has an infinite number of solutions due to the Peano phenomenon, once a trajectory leaves the origin with a radius and an angle, it cannot merge with another trajectory. Furthermore, because the system \eqref{eq:ODE} is autonomous, if for a given $x\in \R^d\smallsetminus\{0\}$ there exists a characteristic $\varphi$ and $t_x$ such that $\varphi(t_x)=x$, then $\varphi_0(t):= \varphi(t+t_0^\varphi)$ is also a characteristic curve which is also extremal and passes through $x$ since $\varphi_0(t_x-t_0^\varphi)=x$. Then, it is enough to define $g$ along the extremal solutions of the Equation \eqref{eq:ODE}. These extremal solutions are uniquely determined by the angle at which they leave $0$, which we will call $\omega_0^{\varphi_0}= \underset{t\to 0^+}{\lim}\frac{\varphi_0(t)}{|\varphi_0(t)|}$.

Now, we prove that if $\varphi_0$ is an extremal  solution of \eqref{eq:ODE}, then it must be a homogeneous function of degree $\frac{1}{1-\gamma}$; i.e. $\varphi_0(\lambda t )= \lambda^{\frac{1}{1-\gamma}} \varphi_0(t)$, $\forall t\geq 0$, $\forall \lambda >0$. Let $\lambda>0$ be fixed and define $\psi_\lambda(t)= \lambda^{-\frac{1}{1-\gamma}} \varphi_0(t)$. We want to prove that $\psi_\lambda(t)=\varphi_0(t)$ $\forall t>0$. Due to the homogeneity of $\nabla U$, we have
\begin{align*}
\dot{\psi}_\lambda(t)& = \lambda^{-\frac{1}{1-\gamma}} \dot{\varphi}_0(\lambda t) \lambda = \lambda^{-\frac{1}{1-\gamma}+1} \nabla U\left(\varphi_0(\lambda t)\right) = 
\lambda^{-\frac{1}{1-\gamma}+1} \nabla U\left(\lambda^{\frac{1}{1-\gamma}}\psi_\lambda( t)\right)\\
& = \lambda^{-\frac{1}{1-\gamma}+1 + \frac{\gamma}{1-\gamma}} \nabla U\left(\psi_\lambda( t)\right) = \nabla U\left( \psi_\lambda(t) \right).
\end{align*}
Then, $\psi_{\lambda}$ is an extremal solution of \eqref{eq:ODE} and
\[
\omega_0^{\psi_\lambda}= \lim_{t\to 0^+} \frac{\psi_\lambda(t)}{\left\vert \psi_\lambda(t)\right\vert}= \lim_{t\to 0^+} \frac{\varphi_0( \lambda t)}{\left\vert \varphi_0( \lambda t)\right\vert}= \omega_0^{\varphi_0},
\] 
then it must be $\psi_\lambda(t)= \varphi_0(t)$ for all $t>0$.

From the homogeneity of the extreme characteristic curves, we can deduce that if we impose that $g(0)=0$, then $g$ defined from these characteristic curves must also be a homogeneous function. If $x\neq 0$ is such that there exists an extremal solution $\varphi_0$ and $t_x$ such that $\varphi_0(t_x)=x$, then we have defined $g(x)=-\lambda_1 t_x$. Let $\lambda>0$ be fixed, then
\[
g(\lambda x)= g\left( \lambda \varphi_0(t_x) \right) = g\left( \varphi_0(\lambda^{1-\gamma} t_x) \right) = -\lambda_1 \lambda^{1-\gamma} t_x = \lambda^{1-\gamma} g(x),
\]
due to the homogeneity of $\varphi_0$.

To conclude, we should note that for the study of the second-order LDP, we are interested only in $x$ belonging to the path of some of the solutions of Equation \eqref{eq:ODE}. However, we can explicitly write $g(x)$ in terms of $x$ when $|x|\to \infty$. For a fixed characteristic curve $\varphi$, let us introduce the functions $r(t)=\left\vert \varphi(t)\right\vert$ and $\omega(t)=\frac{\varphi(t)}{|\varphi(t)|}$ which are respectively the radial and angular components of $\varphi(t)$. Then, $\left(r(t), \omega(t)\right)$ must verify
\[
\begin{cases} \frac{\partial}{\partial t} r = r^\gamma (1+\gamma) \theta(\omega);\\
\frac{\partial}{\partial t} \omega = r^{\gamma-1} \left( \nabla \theta(\omega)-\left\langle \nabla \theta(\omega), \omega\right\rangle \omega \right).\end{cases}
\] 
Let us observe that, using the first equation, we obtain $\frac{\partial}{\partial t} r(t)>0$ since $\theta>0$, which implies that $r(t)\to + \infty$ when $t\to \infty$. Furthermore, this behavior suggests that for large values of $t$, the derivative of $\omega(t)$ is small, meaning that $\omega(t)$ remains close to a constant. In other words, the characteristic curves have an expansive behavior. From the first equation, we get
\[
r(t)=\left[(1+\gamma)\int_{t_0^\varphi}^t \theta(\omega(s))\text{d}s\right]^{\frac{1}{1-\gamma}}=\left[(1+\gamma)\theta( \omega(\hat{s})) (t-t_0^\varphi)\right]^{\frac{1}{1-\gamma}},
\]
for some $\hat{s} \in (t_0^\varphi, t)$. Now, for some $x\in \R^d$, define $t_x$ and $\varphi$ such that $\varphi(t_x)=x$, then
\[
|x|=r(t_x)= \left[(1+\gamma)\theta( \omega(\hat{s}_x)) (t_x-t_0^\varphi)\right]^{\frac{1}{1-\gamma}},
\]
and $(t_x-t_0^\varphi)^+= |x|^{1-\gamma} \frac{1}{(1+\gamma) \theta(\omega(\hat{s}_x))}$. Then, we get
$
g(x)=-\lambda_1 \frac{1}{(1+\gamma) \theta(\omega(\hat{s}_x))} |x|^{1-\gamma}
$, and it can be seen that 
\[
\lim_{|x| \to \infty} \frac{g(x)}{-\frac{\lambda_1}{(1+\gamma) \theta (\frac{x}{|x|})}|x|^{1-\gamma}}=1.
\]

\begin{remark}
 If there exists a function $g$ such that $\left\langle \nabla U(x), \nabla g(x)\right\rangle = -\lambda_1$ $\forall x$, then $\left\langle \nabla U(\lambda x), \nabla g(\lambda x)\right\rangle = -\lambda_1$ $\forall x$, $\forall \lambda>0$, and we have
\[
\left\langle \nabla U(x), \nabla g(x)\right\rangle = \left\langle \nabla U(\lambda x), \nabla g(\lambda x)\right\rangle \Leftrightarrow \left\langle \nabla U(x), \nabla g(x)-\lambda^\gamma \nabla g(\lambda x)\right\rangle =0
\]
for all $x$, for all $\lambda>0$. Since $\nabla U(x)\neq 0$ if $x\neq 0$, then it must to be $\nabla g(\lambda x)= \lambda^{-\gamma}\nabla g(x)$, or $\nabla g(x)-\lambda^\gamma \nabla g(\lambda x) \neq 0$ and it must to be perpendicular to $\nabla U(x)$ for all $x$. We have proved before that if $g(0)=0$, then $g$ defined along the characteristics must be homogeneous, so the first option is verified, i.e., $\nabla g(x)$ must be a homogeneous function of degree $-\gamma$. 
\end{remark}

\begin{remark}
The hypothesis that the drift $b$ comes from a homogeneous potential allows us to establish a relationship between the semigroups 
\[
P_t^{X^\ve}(x)(x)=\E \left[ f(X_t^\ve)|X_0^\ve=x\right] \text{ and } 
T_t^V(f)(x)= \E \left[ f(W_t) e^{-\int_0^t V(W_s)\text{d}s}|W_0=x\right].
\]
We have proved in Lemma \ref{prop:p converges to expU psi1} that the exponential behavior of $X_t^\ve$ depends only on the principal eigenvalue and eigenvector of the linear generator of $T_t^V$, since 
\[
\lim_{\ve \to 0} \ve_\gamma \log\left(p^\ve(t,x)\right)=-\lambda_1 t + \lim_{\ve \to 0} \ve_\gamma \log \left[ e^{U\left(\xe \right)}\psi_1\left(\xe \right)\right],
\]
being $\psi_1$ the ground state of $-\mathcal{L}(f)(x)=-\frac{1}{2}\Delta f(x)+ \frac{1}{2}\left( \left\vert \nabla U(x)\right\vert^2 + \Delta U(x)\right) f(x)$.  Let's define
\[
g_\ve(x):= -\ve_\gamma \log \left[ e^{U\left(\xe \right)}\psi_1\left(\xe \right)\right].
\]
Since $\psi_1$ verifies $-\frac{1}{2}\Delta \psi_1(x)+ \frac{1}{2}\left( \left\vert \nabla U(x)\right\vert^2 + \Delta U(x)\right) \psi_1(x)=\lambda_1 \psi_1(x),$ then $g_\ve(x)$ verifies
\[
-\frac{\ve^2}{2} \Delta g_\ve(x) + \frac{1}{2} \frac{\ve^2}{\ve_\gamma} \left\vert \nabla g_\ve(x)\right\vert^2 + \left\langle \nabla U(x), \nabla g_\ve(x)\right\rangle - \Delta U(\xe)=-\lambda_1.
\]
In particular, $g_\ve$ is a classical solution of the above equation. By letting $\ve \to 0$, we have the following limit equation
\[
\left\langle \nabla U(x), \nabla g(x)\right\rangle = -\lambda_1.
\]
If we could prove that the above equation verifies a Comparison Principle, then we could deduce that the limit  $\underset{\ve \to 0}{\lim} g_\ve (x)$ exists and it is $g(x)$ with $g$ verifying that equation. However, as we mentioned before, we could not prove that this equation verifies the Comparison Principle, so we had to prove the existence of the limit by proving the upper and lower limits from the upper and lower bounds of the ground state $\psi_1$.  Moreover, that is why we prove the existence of such a $g$ using characteristic curves.
\end{remark}

\subsection{Proof of Theorem \ref{thm:Second LDP}} \label{subsection: Second order LDP}

Finally, in this section, we prove the second-order LDP for the family of stochastic processes $\left\{ X^\ve \right\}_\ve$ from the lower and upper bounds obtained for the density $p^\ve(t,x)$. For the proof, we use the following lemma, which reduces the lower and upper bounds of the LDP definition to study the exponential bounds for open and closed balls.

\begin{lemma}\label{lemma:lower-upper LDP}
Let $\left\{ \PP^\ve \right\}_\ve$ a family of probability measures. 
\begin{enumerate}
\item \emph{(Lower LDP)} If for any $x\in \mathcal{X}$ and $\delta >0$, we have
\[
\liminf_{\ve \to 0} \lambda(\ve)^{-1} \log \left(\PP^\ve (B(x, \delta))\right) \geq -I(x) - \mathcal{O}(\delta),
\]
then for each $A\subset \mathcal{X}$ open,
\[
\liminf_{\ve \to 0} \lambda(\ve)^{-1} \log \PP^{\ve}(A) \geq - \inf_{x\in A} I(x). 
\]
\item \emph{(Upper LDP)} If moreover $\{ \PP^\ve \}_\ve$ is exponentially tight, and for any $x\in \mathcal{X}$ and $\delta >0$, we have
\[
\limsup_{\ve \to 0} \lambda(\ve)^{-1} \log \left(\PP^\ve (\overline{B(x, \delta)})\right) \leq -I(x) + \mathcal{O}(\delta),
\]
then for each $C\subset \mathcal{X}$ closed,
\[
\limsup_{\ve \to 0} \lambda(\ve)^{-1} \log \PP^{\ve}(C) \leq - \inf_{x\in C} I(x).
\]
\end{enumerate}
\end{lemma}
For the proof, see, for example, Chapter 4 in \cite{Dembo}.
\bigskip

\begin{proof}[Proof of Theorem \ref{thm:Second LDP}]
The proof follows the same scheme as the proof of Theorem {\bf 3.11} from \cite{Pappalettera} once we have proved the existence of the exponential limit of $p^\ve(t,x)$,
\[
\underset{\ve \to 0}{\lim} \ve_\gamma \log \left(p^\ve (t,x)\right)=-\lambda_1 t-g(x).
\]
We first note that since $\left\{X^\ve \right\}_\ve$ is exponentially tight with rate $\ve^{-2}$ (as proved in Theorem \ref{thm:First LDP}), then it is exponentially tight with rate $\ve^{-1}_\gamma$ since, given $\alpha>0$, there exists a compact $K_\alpha \subset C_0\left([0,T], \R^d \right)$ such that 
$
\underset{\ve \to 0}{\limsup} \, \ve^2 \log \PP \left(X^\ve \notin K_\alpha \right) \leq -\alpha,
$
then
\[
\limsup_{\ve \to 0} \ve_\gamma \log \PP \left(X^\ve \notin K_\alpha \right)= \limsup_{\ve \to 0} \frac{\ve_\gamma}{\ve^2} \ve^2 \log \PP \left(X^\ve \notin K_\alpha \right)=-\infty < -\alpha.
\]
If $\varphi$ is not a solution for Equation \eqref{eq:ODE}, then $I_1(\varphi)>0$, and
\[
\liminf_{\ve \to 0} \ve_\gamma \log \PP \left(X^\ve \in B(\varphi, \delta) \right) \leq \limsup_{\ve \to 0} \frac{\ve_\gamma}{\ve^2} \ve^2 \log \PP\left(X^\ve \in   \overline{B(\varphi, \delta)}\right)=-\infty,
\]
since $\underset{\ve \to 0}{\limsup}\,  \ve^2 \log \PP\left(X^\ve \in   \overline{B(\varphi, \delta)}\right)=-I_1(\varphi)+\mathcal{O}(\delta)$ due to Theorem \ref{thm:First LDP}; i.e., $I_2(\varphi)=+\infty$. If $\varphi$ is a solution for Equation \eqref{eq:ODE}, due to Lemma \ref{lemma:lower-upper LDP}, it is enough to prove that 
\[
\liminf_{\ve \to 0} \ve_\gamma \log \PP\left(\left\Vert X^\ve - \varphi\right\Vert_\infty <\delta\right) \geq -I_2(\varphi)-\mathcal{O}(\delta) \qquad \text{(lower bound),}
\]
and
\[
\limsup_{\ve \to 0} \ve_\gamma \log \PP\left(\left\Vert X^\ve - \varphi\right\Vert_\infty \leq \delta\right) \leq -I_2(\varphi)-\mathcal{O}(\delta) \qquad \text{(upper bound).}
\]
Let be $\delta>0$ and $0<\eta<\delta$, and define the set
\[
\Gamma_{\delta, \eta}=\Gamma_{\delta, \eta}(\varphi):=\left\{ f\in C_0\left([0,t], \R^d\right):\, \left\Vert \varphi-f\right\Vert_\infty \geq \delta; \, \left\vert \varphi(T)-f(T)\right\vert\leq \delta-\eta\right\}.
\]
Since $\Gamma_{\delta,\eta}$ is closed in $C_0\left([0,T], \R^d\right)$, $I_1$ attains its minimum on $\Gamma_{\delta, \eta}$ at a function $f_{*}$. Let us see that $f_{*}$ cannot be a solution of \eqref{eq:ODE}. If $\varphi_1, \varphi_2$ are two different solutions of \eqref{eq:ODE}, then the function $h(t):=\left\vert \varphi_1(t)-\varphi_2(t)\right\vert$ is monotone non-decreasing. Then, if $f_*$ were a solution of \eqref{eq:ODE}, it could not be possible to have $\underset{t\leq T}{\sup}\left\vert \varphi(t)-f_*(t)\right\vert=\left\vert \varphi(T)-f_*(T)\right\vert>\delta$ and $\left\vert \varphi(T)-f_*(T)\right\vert\leq \delta-\eta$. Moreover, since
\[
\PP\left( \left\Vert X^\ve-\varphi \right\Vert_\infty \leq \delta \right) = \PP \left( \left\vert X^\ve_t-\varphi(t) \right\vert \leq \delta \right)
-\PP \left( \left\vert X^\ve_t-\varphi(t) \right\vert \leq \delta;\, \left\Vert X^\ve-\varphi \right\Vert_\infty > \delta \right),
\]
and
\begin{align*}
\PP \left( \left\vert X^\ve_t-\varphi(t) \right\vert \leq \delta;\, \left\Vert X^\ve-\varphi \right\Vert_\infty > \delta \right) & \leq
\PP\left(X^\ve \in \Gamma_{\delta,\eta}\right) + \PP\left(\delta-\eta < \left\vert X^\ve_T -\varphi(T)\right\vert \leq \delta\right)\\
& \leq \PP\left(X^\ve \in \Gamma_{\delta,\eta}\right) + \frac{1}{2} \PP\left( \left\vert X^\ve_T -\varphi(T)\right\vert \leq \delta\right),
\end{align*}
if $\eta$ is sufficiently small, we deduce that
\[
0< \frac{1}{2} \PP\left( \left\vert X^\ve_T -\varphi(T)\right\vert \leq \delta\right)-\PP\left(X^\ve \in \Gamma_{\delta,\eta}\right) \leq 
\PP\left( \left\Vert X^\ve-\varphi \right\Vert_\infty \leq \delta \right) \leq \PP\left( \left\vert X^\ve_T -\varphi(T)\right\vert \leq \delta\right).
\]
Now, we prove the lower bound.
\begin{align*}
\liminf_{\ve \to 0} \ve_\gamma \log \PP \left( \left\Vert X^\ve - \varphi\right\Vert_\infty<\delta \right) & \geq \liminf_{\ve \to 0} \ve_\gamma \log \left[ \frac{1}{2} \PP\left( \left\vert X^\ve_T -\varphi(T)\right\vert < \delta\right)-\PP\left(X^\ve \in \Gamma_{\delta,\eta}\right)\right] \\
&= \max\left\{ \liminf_{\ve \to 0}  \ve_\gamma \log \PP\left( \left\vert X^\ve_T -\varphi(T)\right\vert \leq \delta\right); \, \liminf_{\ve \to 0}  \ve_\gamma \log \PP\left(X^\ve \in \Gamma_{\delta,\eta}\right)\right\}.
\end{align*}
Since $\underset{\ve \to 0}{\lim} \ve^2 \log \PP \left( X^\ve \in \Gamma_{\delta, \eta}\right)=-I_1(f_*)<0$, then 
\[
 \liminf_{\ve \to 0}  \ve_\gamma \log \PP\left(X^\ve \in \Gamma_{\delta,\eta}\right) =  \liminf_{\ve \to 0}  \frac{\ve_\gamma}{\ve^2} \ve^2 \log \PP\left(X^\ve \in \Gamma_{\delta,\eta}\right)=-\infty.
\]
Finally,
\begin{align*}
\liminf_{\ve \to 0} \ve_\gamma \log \PP \left( \left\Vert X^\ve - \varphi\right\Vert_\infty<\delta \right) & \geq 
\liminf_{\ve \to 0} \ve_\gamma \log \PP \left( | X^\ve_T - \varphi(T)|<\delta \right)\\
& = \liminf_{\ve \to 0} \ve_\gamma \log \int_{B(\varphi(T), \delta)}p^\ve (T,x) \text{d}x\\
& = \liminf_{\ve \to 0} \ve_\gamma \log\Big[ \text{vol}\left(B(\varphi(T), \delta)\right) \\
& \times \int_{B(\varphi(T), \delta)}\frac{p^\ve (T,x)}{\text{vol}\left(B(\varphi(T), \delta)\right)} \text{d}x \Big]\\
& \geq \int_{B(\varphi(T), \delta)} \frac{-\lambda_1 T-g(x)}{\text{vol}\left(B(\varphi(T), \delta)\right)} \text{d}x \\
& \to_{\delta} -\lambda_1T - g(\varphi(T)):= -I_2(\varphi)
\end{align*}
\bigskip
Due to the exponential tightness, the upper LDP is proved analogously.
\end{proof}

As a corollary, we deduce that the family of processes $\left\{X^\ve \right\}_\ve$ converges to the set of extremal solutions of Equation \eqref{eq:ODE} since
\[
I_2(\varphi)=\lambda_1 T + g\left(\varphi(T)\right)=\lambda_1 T +g(0)-\lambda_1\left(T-t_0^{\varphi}\right)^+ = 0,
\]
if $t_0^\varphi = 0$, and $I_2(\varphi) >0$ if $t_0^\varphi> 0$ ($g(0)=0$ due to the homogeneity of $g$).

\section{Final comments and directions for future work}\label{section:final comments}
Finally, in this section, we briefly describe the study of large deviations for the Peano phenomenon presented in \cite{Herrmann} and \cite{Pappalettera}, and explain how these results can be interpreted in the context of our work. Moreover, we present some comments on extending the results of this paper.
\bigskip

In \cite{Herrmann}, a study of large deviations is performed for the Peano phenomenon in the particular case where $d=1$ and the drift is of the form $b(x)=\text{sgn}(x)|x|^{\gamma}=x |x|^{\gamma-1}$. For this case, two extremal solutions exist for the ODE $\dot{x}=b(x);$ $x(0)=0$, and they are calculated explicitly as $\varphi_1(x)=\left((1-\gamma)t\right)^{\frac{1}{1-\gamma}}$ and $\varphi_2(x)=-\left((1-\gamma)t\right)^{\frac{1}{1-\gamma}}$. Next, if $p^\ve(t,x)$ is the density of the random variable $X^\ve_t$, it is noted that the behavior of $p^\ve(t,x)$ is different depending on whether or not $(t,x)$ is in the region enclosed by the graphs of $\varphi_1$ and $\varphi_2$.

\begin{itemize}
\item If the point $(t,x)$ is outside the region enclosed by the graphs of $\varphi_1$ and $\varphi_2$, there exists a positive function $k_t$ such that ${\displaystyle \lim_{\ve \to 0} \ve^2 \log \left(p^\ve (t,x)\right) = -k_t(|x|)}$. Then, the density has an exponential decay with rate $\ve^2$, and the rate is the same as in the case of \cite{F&W} when the dynamical system has a unique solution. 
\item If the point $(t,x)$ lies in the domain between the graphs of  $\varphi_1$ and $\varphi_2$, then it is proved that ${\displaystyle \lim_{\ve \to 0} \ve^2 \log \left(p^{\ve} (t,x)\right) \leq 0}$, and the density has an exponential decay with a different rate, namely $\ve_{\gamma}=\ve^{2\frac{1-\gamma}{1+\gamma}}$. Precisely, it is proved that $\underset{\ve \to 0}{\lim} \ve_{\gamma} \log \left(p^{\ve}(t,x)\right) = \lambda_1 \left( \frac{|x|^{1-\gamma}}{1-\gamma} - t \right),$
where $\lambda_1$ is the first positive eigenvalue of the Schrödinger operator $-\frac{1}{2}\frac{\partial^2}{\partial^2 x}+ \frac{\gamma}{2|x|^{1-\gamma}} + \frac{|x|^{2\gamma}}{2}$.  
\end{itemize} 

Now, from our work, it is relatively straightforward to interpret this result: if $(t,x)$ is in the region enclosed by $\varphi_1$ and $\varphi_2$, then there exists a (non-extremal) solution $\varphi$ of the ODE such that $\varphi(t)=x$ and we know that $I_1(\varphi)=0$, being $I_1$ the rate of the first-order LDP. We further know that the rate of the second-order LDP for that solution is $I_2(\varphi)= \lambda_1 T -\lambda_1\left(T-t_0^\varphi\right)^+$. Note also that for $d>1$, it would be impossible to distinguish between the regions enclosed by the (infinite) extremal solutions, so it would not be possible to study the exponential behavior of the density $p^\ve(t,x)$ according to the location of the point $(t,x)$. This is why a study of large deviations in which the LD rate is defined for the possible limiting trajectories of $X^\ve$ is critical for generalizing this work to more general drift functions. 

On the other hand, the proof of the large deviation for the density $p^\ve(t,x)$ for $(t,x)$ enclosed between the extremal solutions makes essential use of the explicit viscosity solution $u(t,x)$ for the following Hamilton-Jacobi equation:
\begin{equation}\label{eq:H-J-Hermann}
\frac{\partial}{\partial t} u+ H\left(x, \frac{\partial}{\partial x} u\right)=0,
\end{equation}
where the Hamiltonian is $H(x,p)=x^\gamma p$. This equation comes from considering the following representation of the density (see Corollary 1 from \cite{Herrmann})
\[
p^\ve(t,x)= \frac{1}{\ve \sqrt{2\pi t}} \exp\left\{ \frac{|x|^{\gamma + 1}}{\ve^2 (\gamma + 1)} - \frac{x^2}{2\ve^2 t}\right\} \times \E_{\xe} \left[ \exp\left\{ -\int_0^{\frac{t}{\ve_\gamma}} \frac{V(W_s)}{2} \text{d}s\right\} \big| W_{\frac{t}{\ve_\gamma}}=0 \right],
\]
in terms of the Shrödinger semigroup
\[
T_t(f)(x)=\E_x \left[ f(W_t) \exp\left\{ -\frac{1}{2} \int_0^t V(W_s) \text{d}s\right\}\right]
\]
for the potential $V(x)=\frac{\gamma}{|x|^{1-\gamma}}+ |x|^{2\gamma}$, and the Rosenblatt theorem (see \cite{Rosenblatt} ), which states that if $V$ is bounded below and
\[
Q(t,x)=\frac{1}{(2\pi t)^\frac{d}{2}} e^{-\frac{|x|^2}{2t}} \E \left[ e^{-\frac{1}{2} \int_0^t V(W_s)\text{d}s }  \big| W_t=x \right],
\]
then $Q(t,x)$ is solution of $\frac{\partial}{\partial t}Q(t,x)=\frac{1}{2}\Delta Q(t,x)-\frac{1}{2}V(x)Q(t,x)$. Then, it can be proved that $\mu^\ve(t,x):=-\ve_\gamma \log\left( p^\ve (t,x) + e^{-\frac{K}{\ve_\gamma}}\right)$ is a clasical solution of 
\[
\frac{\partial}{\partial t} \mu^\ve + H^\ve \left(t, x, \mu^\ve, \frac{\partial}{\partial x} \mu^\ve, \frac{\partial^2}{\partial^2 t} \mu^\ve\right)=0
\]
for a Hamiltonian $H^\ve$ converging to $H$. Then, the limit ${\displaystyle \lim_{\ve \to 0} \ve_\gamma \log \left(p^\ve(t,x)\right)}$ is obtained from $u(t,x)$ which is the limit of $\mu^\ve(t,x)$ when $\ve \to 0$. This is a powerful analytical approach. However, the need to explicitly know $\mu(t,x)$ is also the main restriction in extending the theory to higher dimensions. In higher dimensions, this approach fails unless special symmetries allow a dimensionality reduction, as is done in \cite{Pappalettera}. Moreover, proving that $u(t,x)$ is the unique viscosity solution of the Hamilton-Jacobi equation \eqref{eq:H-J-Hermann} presents a great difficulty for the general case, since at least we cannot prove that this equation verifies the Comparison Principle by using the classic tools for proving this type of uniqueness.

\bigskip

Finally, we present some comments on extending the results of this paper.

The first question that naturally arises is whether it is possible to study an LDP for an even slower speed than $\ve_\gamma$, which would make it possible to identify a favorite among the set of extreme solutions. In  \cite{pilipenko2025} it is conjectured that the most likely extreme solutions should be those for which  occurs:
\[
\theta\left(\omega_0^{\varphi_0}\right) = \underset{z\in \mathbb{S}^{d-1}}{\sup} \theta(z). 
\]
Will it be possible to prove this from a study of large deviations of a higher order?
\bigskip

In another line of research, we assume that Equation \eqref{eq:ODE} presents a single Peano point. Since $b(x)=\nabla U(x)=\theta_1(\frac{x}{|x|})|x|^\gamma$ and $\theta_1$ is such that $|\theta_1(y)|^2>(\theta(y))^2$ for all $y \in \mathbb{S}^{d-1}$, then $b$ only cancels at $0\in \R^d$ if we assume that $\theta$ is a strictly positive function on $\mathbb{S}^{d-1}$. Future work could determine what happens when Equation \eqref{eq:ODE} has more than one Peano point, that is, whether $\theta_1$ can be canceled on $\mathbb{S}^{d-1}$. 
\bigskip

Another possible line of research could be to analyze the case in which $b(x)=\nabla U(x)$; however, $U$ is not homogeneous. That is, to analyze the case where the semigroup $P_t^{X^\ve}(f)(x)$ is related to $T_t(f)(x)=\E \left[ f(W_t) e^{-\int_0^t V^{\ve} (W_s)\text{d}s}\right]$, but the potential $V^\ve$ depends on $\ve$. What happens to the generator $-\mathcal{L}^{\ve}$? Will it still be true that the only eigenvector influencing the large deviation is $\psi_1^\ve$?
\bigskip

More generally, what happens when $b$ does not come from a potential? One possible way for further research could be to follow the strategy proposed by \cite{feng2014large}, which is based on the study of the convergence of nonlinear semigroups associated with the family $\left\{X^\ve\right\}_\ve$. As was mentioned before, the main difficulty we find with this strategy is proving the uniqueness of solutions to the Hamilton-Jacobi equations involved.



\newpage
\bibliographystyle{alea3}
\bibliography{References}

\end{document}